\documentclass[final,onefignum,onetabnum]{siamart190516}

\usepackage{enumitem}
\usepackage{tikz}

\usepackage{lipsum}
\usepackage{amsfonts,amsmath,amssymb}
\usepackage{graphicx}
\usepackage{epstopdf}
\usepackage{algorithmic}
\ifpdf
  \DeclareGraphicsExtensions{.eps,.pdf,.png,.jpg}
\else
  \DeclareGraphicsExtensions{.eps}
\fi

\usetikzlibrary{patterns}
\newsiamremark{remark}{Remark}
\newsiamremark{hypothesis}{Hypothesis}
\crefname{hypothesis}{Hypothesis}{Hypotheses}
\newsiamthm{claim}{Claim}

\headers{The nonlocal Neumann problem}{L. Frerick, C. Vollmann and M. Vu}

\title{The nonlocal Neumann problem\thanks{This work has been supported by the German Research Foundation (DFG) within the Research Training Group 2126: “Algorithmic Optimization”.}}

\author{Leonhard Frerick\thanks{Department of Mathematics, Trier University, Universitätsring 15, Trier, 54296, Germany
  (\email{frerick@uni-trier.de, vollmann@uni-trier.de, vuv@uni-trier.de}).}
\and Christian Vollmann\footnotemark[2]
\and Michael Vu\footnotemark[2]}

\usepackage{amsopn}

\newcommand{\Rd}{\mathbb{R}^{d}} % R hoch d

\newcommand{\kernel}{\mathcal{K}}%Space of all nonnegative kernel functions (non symmetic as well)

\newcommand{\V}{\mathrm{V}(\Omega;\gamma)}% the quotient space of our test space
\newcommand{\Vd}{\mathrm{V}_{0}(\Omega;\gamma)}% Dirichlet space of our test space

\newcommand{\Tr}{\mathrm{Tr}}% Trace operator

\newcommand{\Ls}{{\mathrm{L}^{2}(\Omega)}}

\newcommand{\einf}{ \operatorname*{ess\,\inf}}% essentielle inf
\newcommand{\esup}{ \operatorname*{ess\,\sup}}% essentielle dup

\newcommand{\La}{\mathcal{L}}

\newcommand{\Neu}{\mathcal{N}}

\ifpdf
\hypersetup{
  pdftitle={The nonlocal Neumann problem},
  pdfauthor={L. Frerick, C. Volmann and M. Vu}
}
\fi

\begin{document}

\maketitle

% REQUIRED
\begin{abstract}
The classical local Neumann problem is well studied and solutions of this problem lie, in general, in a Sobolev space. In this work, we focus on \textit{nonlocal} Neumann problems with measurable, nonnegative kernels, whose solutions require less regularity assumptions. For kernels of this kind we formulate and study the weak formulation of the nonlocal Neumann problem and we investigate a nonlocal counterpart of the Sobolev space $H^{1}$ as well as a resulting nonlocal trace space. We further establish, mainly for symmetric kernels, various existence results for the weak solution of the Neumann problem and we discuss related necessary conditions. Both, homogeneous and nonhomogeneous Neumann boundary conditions are considered. In addition to that, we present a new weak formulation of a Robin problem, where we reformulate the Robin problem into a ``regional'' problem.
\end{abstract}

% REQUIRED

% REQUIRED

\section{Introduction and results}\label{sec1}
Long and short range interactions are modeled with the help of nonlocal instead of local differential operators, so by considering nonlocal interactions between sets, whose topological closures are possibly disjoint. A nonlocal convection–diffusion operator $\La$ is typically defined as an integral operator of the form
\begin{equation*}
    \mathcal{L}u(x):=\mathcal{L}_{\gamma}u(x):=\int_{\Rd}u(x)\gamma(x,y)-u(y)\gamma(y,x)\,\mathrm{d}y\quad\text{for }  x\in\Rd,
\end{equation*}
where $\gamma\colon\Rd\times\Rd\to\mathbb{R}$ and $u\colon\Rd\to\mathbb{R}$ are measurable functions and where we may assign the Cauchy principal value to the integral, if necessary. 
In this article, we only consider Lebesgue measurable functions and sets and we denote the $d-$dimensional Lebesgue measure by $\lambda$. 

The operator $\La$ can be regarded a nonlocal counterpart of the local convection–diffusion operator (see, e.g, \cite{gunz_1})
\begin{equation*}
    \mathcal{E}u(x):=\operatorname{div}(A(x)\nabla u(x)-b(x) u(x))+c(x)^{\top}\nabla u(x)\quad\text{for } x\in\Rd,
\end{equation*}
where $A\colon\Rd\to\mathbb{R}^{d\times d}$, $b\colon\Rd\to\Rd$ and $c\colon\Rd\to\Rd$ are functions, $\nabla u$ the gradient of the smooth function $u\colon\Rd\to\mathbb{R}$ and $\operatorname{div}$ is the local divergence operator.

Due to the nonlocal nature of the operator $\mathcal{L}$, it is generally not enough to impose boundary conditions just on $\partial\Omega$, the topological boundary of $\Omega$, in order to guarantee well-posedness of related nonlocal boundary value problems. For instance, given a bounded and open set $\Omega\subset\Rd$ and functions $f\colon \Omega\to\mathbb{R},~g\colon \Rd\setminus\Omega\to\mathbb{R}$, a nonlocal Dirichlet problem can be defined by % to find a measurable function  $u\colon\Rd\to\mathbb{R}$ with 
\begin{equation}\label{dir}
    \left\{\begin{aligned}
        \mathcal{L}u&=f\quad \text{on }\Omega,\\
        u&=g\quad \text{on }\Rd\setminus\Omega.
    \end{aligned}\right.
\end{equation}
Problems of this kind are studied for example by Felsinger et al.\,\cite{fels} and Du et al.\,\cite{gunz_1}. For problems with a finite range of nonlocal interactions, Du et al.\,\cite{gunz_1} have shown, that it already suffices to impose the Dirichlet boundary condition only on a certain subset $\Omega_{I}\subset\Rd\setminus\Omega$ for establishing well--posedness results of problem \eqref{dir}. 

Du et al.\,\cite{gunz_1} have also introduced a nonlocal Neumann operator and the corresponding nonlocal Neumann problem
\begin{equation}\label{du}
    \left\{\begin{aligned}
        \mathcal{L}u(x)&=\int_{\Rd}(u(x)-u(y))\gamma(y,x)\,\mathrm{d}y&=f(x)\quad &\text{for }x\in\Omega,\\
        \Neu_{D} u(y)&=\int_{\Omega\cup\Omega_{I}}(u(x)-u(y)\gamma(y,x))\,\mathrm{d}x&=g(y)\quad &\text{for }y\in\Omega_{I},
    \end{aligned}\right.
\end{equation}
where the Neumann--type boundary condition is imposed on the subset $\Omega_{I}\subset\Rd\setminus\Omega$, the so-called \textit{interaction domain}, given by
\begin{equation*}
    \Omega_{I}:=\{y\in\Rd\setminus\Omega\text{ such that } \gamma(x,y)\neq0 \text{ for some }x\in\Omega\}.
\end{equation*}
Besides this, various approaches to nonlocal Neumann problems have been investigated in the literature. For instance, Barles et al.\,\cite{vis_1} formulate a Neumann--type boundary value problem by imposing a reflection condition on $\Rd\setminus\Omega$. Cortazar et al.\,\cite{CORTAZAR2007360} derive a Neumann--type boundary value problem from a decomposition of the operator $\La$. Dipierro et al.\,\cite{dip} introduce and study the Neumann operator
\begin{equation}\label{def:NeuOpDip}
    \Neu_{s}u(y)=\int_{\Omega}\frac{u(y)-u(x)}{\|x-y\|^{d+2s}}\,\mathrm{d}x\quad\text{for } s\in(0,1), ~ y\in\Rd\setminus\Omega, 
\end{equation}
and the resulting fractional Neumann problem 
\begin{equation}\label{frac}
    \left\{\begin{aligned}
        \mathcal{L}_{s}u(x)&=\int_{\Rd}\frac{u(x)-u(y)}{\|x-y\|^{d+2s}}\,\mathrm{d}y&=f(x)\quad &\text{for }x\in\Omega,\\
        \Neu_{s} u(y)&=\int_{\Omega}\frac{u(y)-u(x)}{\|x-y\|^{d+2s}}\,\mathrm{d}x&=g(y)\quad &\text{for }y\in\Rd\setminus\Omega.
    \end{aligned}\right.
\end{equation}
With the function spaces introduced by Felsinger et al.\,\cite{fels}, Foghem \cite{fog} and Foghem and Kaßmann \cite{kass} study the nonlocal Neumann problem
\begin{equation}\label{levy}
    \left\{\begin{aligned}
        \mathcal{L}_{\nu}u(x)&=\int_{\Rd}(u(x)-u(y))\nu(x-y)\,\mathrm{d}y&=f(x)\quad &\text{for }x\in\Omega,\\
        \Neu_{\nu} u(y)&=\int_{\Omega}(u(x)-u(y))\nu(x-y)\,\mathrm{d}x&=g(y)\quad &\text{for }y\in\Rd\setminus\Omega,
    \end{aligned}\right.
\end{equation}
where $\nu\colon\Rd\setminus\{0\}\to[0,\infty]$ is even and Lévy integrable, i.e., $\nu$ satisfies
\begin{equation*}
    \nu(x)=\nu(-x)\quad \text{for }\Rd\setminus\{0\}\text{ and } \int_{\Rd}\min\{1,\|x\|^{2}\}\nu(x)\,\mathrm{d}x<\infty.
\end{equation*}
Note that for every $s\in(0,1)$ the fractional Laplace kernel is even and Lévy integrable.

We want to remark that, while our studies of the nonlocal Neumann problem were independent of the studies in \cite{fog} and \cite{kass}, some results are similar. In particular the nonlocal integration by parts formula (see \Cref{thm:gauss} and Proposition 4.8 in \cite{fog}), the resulting Hilbert space (see \Cref{corollary:hilb} and see 3.46 in \cite{fog}) and the existence result (see \Cref{thm:exriesz} and Theorem 4.16 in \cite{fog}).

The goal of this article is to connect the Neumann operator \eqref{def:NeuOpDip} introduced by Dipierro et al.\,\cite{dip} and the function spaces introduced by Felsinger et al.\,\cite{fels} with the approach of Du et al.\,\cite{gunz_1} to impose boundary conditions only on a subset of $\Rd\setminus\Omega$. Specifically, we study the weak formulation of the nonlocal Neumann problem
\begin{equation}\tag{P}\label{NP}
    \left\{\begin{aligned}
        \mathcal{L}u(x)&:=\mathcal{L}_{\gamma}u(x)&=\int_{\Rd}u(x)\gamma(x,y)-u(y)\gamma(y,x)\,\mathrm{d}y&=f(x)\quad &\text{for }x\in\Omega,\\
        \Neu u(y)&:=\Neu_{\gamma} u(y)&=\int_{\Omega}u(y)\gamma(y,x)-u(x)\gamma(x,y)\,\mathrm{d}x&=g(y)\quad &\text{for }y\in\widehat{\Gamma},
    \end{aligned}\right.
\end{equation}
where 
\begin{equation*}
    \gamma\in\kernel:=\{k\colon\Rd\times\Rd\to[0,\infty)\text{ is measurable}\}
\end{equation*}
is referred to as \textit{kernel} and the \textit{nonlocal boundary} $\widehat{\Gamma}$ is given by
\begin{equation*}
    \widehat{\Gamma}:=\widehat{\Gamma}(\Omega,\gamma):= \Big\{y\in\Rd\setminus\Omega\colon\int_{\Omega}\gamma(y,x)+\gamma(x,y)\mathrm{d}x>0\Big\}.
\end{equation*}
If the kernel $\gamma\in\kernel$ vanishes identically in the complement of $\Omega\times\Omega$, then we call $\gamma$ \textit{regional} and we say that problem \cref{NP} is a \textit{regional problem}. Note that we impose the boundary condition in problem \eqref{NP} on $\widehat{\Gamma}$, a possibly strict subset of $\Rd\setminus\Omega$.

As in \cite{dip}, \cite{fog} and \cite{kass}, a fundamental aspect of dealing with problem \eqref{NP} is a variational formulation facilitated by the nonlocal integration by parts formula
\begin{align*}
    &\int_{\Omega}\La u(x)\,v(x)\,\mathrm{d}x+\int_{\widehat{\Gamma}}\Neu u(y)\,v(y)\,\mathrm{d}y\\
    =&\frac{1}{2}\int_{\Omega}\int_{\Omega}(u(x)\gamma(x,y)-u(y)\gamma(y,x))(v(x)-v(y))\,\mathrm{d}y\,\mathrm{d}x\\
    &+\int_{\Omega}\int_{\widehat{\Gamma}}(u(x)\gamma(x,y)-u(y)\gamma(y,x))(v(x)-v(y))\,\mathrm{d}y\,\mathrm{d}x,
\end{align*}
for sufficiently regular measurable functions $u,v\colon \Rd\to\mathbb{R}$. We prove this formula in the next section.  %Through this formula, we define and study the variational weak solution and prove several existence results. 

Finally, let us provide a physical interpretation of problem \eqref{NP} by following along the lines of \cite{nolocal_vector_calculus}. A probabilistic perspective on the problem at hand can be found in \cite{dip}. Here, we understand the Neumann operator $\Neu$ as a nonlocal flux density. More precisely, for a measurable function $u\colon\Rd\to\mathbb{R}$ and a point $y\in\Rd$, the \textit{flux density from $\Omega$ to $y$} is given by
\begin{align*}
    \Neu u (y)=\int_{\Omega}u(y)\gamma(y,x)-u(x)\gamma(x,y)\,\mathrm{d}x.
\end{align*}
%{\color{cyan}\Neu u(y)}=\int_{\Omega}u(y)\gamma(y,x)-u(x)\gamma(x,y)\,\mathrm{d}x.
% \begin{figure}[H]
%    \begin{minipage}[H]{0.55\textwidth}
%        \centering
%     	\begin{align*}
%     	    {\color{cyan}\Neu u(y)}=\int_{\Omega}u(y)\gamma(y,x)-u(x)\gamma(x,y)\,\mathrm{d}x.
%        \end{align*}
%    \end{minipage}
%    \begin{minipage}[H]{0.4\textwidth}
%            \centering
%            \begin{tikzpicture}[yscale=0.15,xscale=0.3]
%        		\fill [black!10] (0, 0) ellipse (8 and 6) ;
%        		\fill [black!20] (0, 0) ellipse (6 and 4) node[black] {$\Omega$};
%        		\node (y) at (6.8,2.4) {};
%        		\node () at (1.5,-5) {$\widehat{\Gamma}$};
%        		\fill (y) circle (4pt)  node[right,black] {$y$};
%        		\draw[<-,line width=1pt,cyan] (6.7,2.4) to[in=35,out=180] (-5,0) node[right]{}; 
%        		\draw[<-,line width=1pt,cyan] (6.8,2.2) to[in=45,out=246] (1,-2) node[right]{}; 
%            \end{tikzpicture}
%            \caption{Flux density from $\Omega$ to $y\in\widehat{\Gamma}$}
%      \end{minipage}
%\end{figure}
Now, by integrating $\Neu u(y)$ over a measurable set $A\subset\Rd$ we obtain the \textit{flux from $\Omega$ into $A$}, namely,
%{\color{cyan}\eu u(y)}
%{\color{red!70}A}
% \begin{figure}[H]
%	\begin{minipage}[H]{0.3\textwidth}
%		\centering
%		\begin{equation*}
%		{\color{red!70}\int_{A}}{\color{cyan}\eu u(y)}\,{\color{red!70}\mathrm{d}y}.
%		\end{equation*}
%	\end{minipage}
%	\begin{minipage}[H]{0.6\textwidth}
%		\centering\begin{tikzpicture}[xscale=0.8,yscale=0.6]
%		\fill [black!20] (0, 0) ellipse (2 and 1) node[black] {$\Omega$};
%		\fill[color=red!60](5,0)ellipse (20pt and 7pt);
%		\draw [red!70](4,-0.2) node{$A$};
%		\draw[->, thick,color=cyan] (1.75,0.4) to[out=15, in=135] (4.5,0.25);
%		\end{tikzpicture}
%		\caption{Flux from $\Omega$ to ${\color{red!70}A}$}
%	\end{minipage}
%\end{figure}
\begin{equation*}
    \int_{A}\Neu u(y)\,\mathrm{d}y.
\end{equation*}
%{\color{red!70}\int_{A}}{\color{cyan}\Neu u(y)}\,{\color{red!70}\mathrm{d}y}.
An illustration is provided in Figure \ref{fig:fluX}.
%%%%
\begin{figure}\label{fig:fluX}
\centering
\begin{minipage}{0.45\textwidth}
        \begin{tikzpicture}[yscale=0.15,xscale=0.3]
    		\fill [black!10] (0, 0) ellipse (8 and 6) ;
    		\fill [black!20] (0, 0) ellipse (6 and 4) node[black] {$\Omega$};
    		\node (y) at (6.8,2.4) {};
    		\node () at (1.5,-5) {$\widehat{\Gamma}$};
    		\fill (y) circle (4pt)  node[right,black] {$y$};
    		\draw[<-,line width=1pt,dotted] (6.7,2.4) to[in=35,out=180] (-5,0) node[right]{}; 
    		\draw[<-,line width=1pt,dotted] (6.8,2.2) to[in=45,out=246] (1,-2) node[right]{}; 
        \end{tikzpicture}
\end{minipage}
\hspace*{0.5cm}
	\begin{minipage}{0.45\textwidth}
\begin{tikzpicture}[xscale=0.75,yscale=0.6]
    		\fill [black!20] (0, 0) ellipse (2 and 1) node[black] {$\Omega$};
            \fill[](5,0)ellipse (20pt and 7pt);
            \draw [](4,-0.2) node{$A$};
            \draw[->, thick,,dotted] (1.75,0.4) to[out=15, in=135] (4.5,0.25);
    \end{tikzpicture}
\end{minipage}
\caption{Flux density from $\Omega$ to $y\in\widehat{\Gamma}$ (left) and Flux from $\Omega$ to $A$ (right).}
\end{figure}
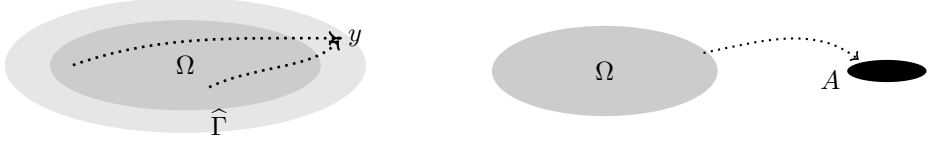
%%%
Thereby it follows that the nonlocal boundary $\widehat{\Gamma}$ consists of those points in $\Rd\setminus\Omega$, for which a flux from $\Omega$ is possible.

The remainder of this article is organized as follows. In Section \ref{sec:var}, we define the test function space $\V$, the nonlocal counterpart of the Sobelev space $\mathrm{H}^{1}(\Omega)$, and show that it is a separable Hilbert space. In Section \ref{ssec:gaus}, we prove a nonlocal integration by parts formula. And in Section \ref{ssub:sym}, we study the resulting weak formulation of the Neumann problem for symmetric kernels (see \Cref{thm:exriesz} and \Cref{thm:exnonhom}) and remark on the more demanding nonsymmetric case (see \Cref{rem:nons}). 
% In any case the main focus of Section \ref{sec:var} are the existence results of a weak solution for a symmetric kernel . 
%The nonsymmetric case requires more work, because the bilinear form is in that case, in general, not positive semidefinit, let alone coercive. So we require stronger assumptions in that case (see \Cref{rem:nons}). 
While Section \ref{sec:poin} is dedicated to the nonlocal Poincar\'e  inequality, we introduce a nonlocal trace space in Section \ref{sec:tr}. Finally, in Section \ref{sec:reg}, we discuss Robin boundary conditions, where we incorporate the boundary condition into the kernel to obtain a regional problem. 

%%================================%%

\section{Variational analysis}\label{sec:var}

\subsection{Function space}\label{ssec:fun}

For a nonempty, open set $\Omega\subset\Rd$ and a kernel $\gamma\in\kernel$, we define the function space which was first introduced in \cite{fels},
\begin{align*}
    \mathcal{V}(\Omega;\gamma)&:=\left\{v\colon \Rd\to\mathbb{R} \text{ measurable }\colon\|v\|_{\mathcal{V}(\Omega;\gamma)}<\infty\right\}, \,\text{where}\\
    \|v\|_{\mathcal{V}(\Omega;\gamma)}^{2}&:=\int_{\Omega}v^{2}(x)+\int_{\Rd}(v(x)-v(y))^{2}\gamma(y,x)\,\mathrm{d}y\,\mathrm{d}x.
\end{align*}
We further define a nonlocal boundary by
\begin{equation}\label{eq:DefinitionGamma}
    \Gamma:=\Gamma(\Omega,\gamma):=\{y\in\Rd\setminus\Omega\colon\int_{\Omega}\gamma(y,x)\mathrm{d}x>0\}.
\end{equation}
Then for $\gamma^{\top}(x,y):=\gamma(y,x)$, we have $\widehat{\Gamma}(\Omega,\gamma)=\Gamma(\Omega,\gamma)\cup\Gamma(\Omega,\gamma^{\top})$. Especially for a symmetric $\gamma\in\kernel$, i.e., $\gamma=\gamma^{\top}$, we obtain $\widehat{\Gamma}(\Omega,\gamma)=\Gamma(\Omega,\gamma)=\Gamma(\Omega,\gamma^{\top})$. We define the symmetric bilinear form $\mathfrak{B}\colon \mathcal{V}(\Omega;\gamma)\times \mathcal{V}(\Omega;\gamma)\to \mathbb{R}$ by
\begin{align*}
	\mathfrak{B}(u,v):=\mathfrak{B}_{\gamma}(u,v):=&\frac{1}{2}\int_{\Omega}\int_{\Omega}(u(x)-u(y))(v(x)-v(y))\,\gamma(y,x)\,\mathrm{d}y\,\mathrm{d}x\\
	&+\int_{\Omega}\int_{\Gamma}(u(x)-u(y))(v(x)-v(y))\,\gamma(y,x)\,\mathrm{d}y\,\mathrm{d}x,
\end{align*}
and the mapping $\langle \cdot,\cdot\rangle_{\mathcal{V}(\Omega;\gamma)}\colon \mathcal{V}(\Omega;\gamma)\times \mathcal{V}(\Omega;\gamma)\to\mathbb{R}$,
\begin{align*}
    \langle u,v\rangle_{\mathcal{V}(\Omega;\gamma)}:= \int_{\Omega}u(x)v(x)+\int_{\Omega \cup\Gamma}(u(x)-u(y))(v(x)-v(y))\gamma(y,x)\,\mathrm{d}y\,\mathrm{d}x.
\end{align*}
We observe that $\langle \cdot,\cdot\rangle_{\mathcal{V}(\Omega;\gamma)}$ is a semi-inner-product and therefore $\langle\cdot,\cdot\rangle_{\mathcal{V}(\Omega;\gamma)}$ induces a seminorm $\|\cdot\|_{\mathcal{V}(\Omega;\gamma)}$ on $\mathcal{V}(\Omega;\gamma)$. Furthermore, $\mathfrak{B}$ is bounded by definition and for all functions $u,v\in \mathcal{V}(\Omega;\gamma)$ we have 
\begin{equation}\label{equ:equiv}
    \frac{1}{2} \langle  u,u\rangle_{\mathcal{V}(\Omega;\gamma)}\leqslant \int_{\Omega}u^{2}(x)\,\mathrm{d}x + \mathfrak{B}(u,u)\leqslant \langle  u,u\rangle_{\mathcal{V}(\Omega;\gamma)}
\end{equation}
 and
 \begin{equation*}
    \begin{split}
        &\int_{\Omega}\int_{\Gamma}(u(x)-u(y))(v(x)-v(y))\,\gamma(y,x)\,\mathrm{d}y\,\mathrm{d}x\\
        =&\int_{\Omega}\int_{\Rd\setminus\Omega}(u(x)-u(y))(v(x)-v(y))\,\gamma(y,x)\,\mathrm{d}y\,\mathrm{d}x.
    \end{split} 
\end{equation*}
By following the argumentation made in \cite[Proposition 3.1]{dip}, \cite[Lemma 2.3]{fels}, \cite[Therorem 3.1]{wloka} and \cite[Theorem 3.46]{fog}, we obtain:
\begin{theorem}\label{thm:comp}
    Let $\Omega\subset \Rd$ be a nonempty and open and $\gamma\in\kernel$. Then $\mathcal{V}(\Omega;\gamma)$ is a complete seminormed vector space with respect to $\|\cdot\|_{\mathcal{V}(\Omega;\gamma)}$ and $ \langle \cdot,\cdot\rangle_{\mathcal{V}(\Omega;\gamma)}$ is a semi-inner-product on $\mathcal{V}(\Omega;\gamma)$.
\end{theorem}
\begin{proof}
    
    As mentioned before, $(\mathcal{V}(\Omega;\gamma),\|\cdot\|_{\mathcal{V}(\Omega;\gamma)})$ is a seminormed vector space, because $ \langle \cdot,\cdot\rangle_{\mathcal{V}(\Omega;\gamma)}$ is bilinear and positive semi-definite, i.e.,\, $ \langle \cdot,\cdot\rangle_{\mathcal{V}(\Omega;\gamma)}$ is a semi-inner-product on $\mathcal{V}(\Omega;\gamma)$.
    
    In order to show that $(\mathcal{V}(\Omega;\gamma),\|\cdot\|_{\mathcal{V}(\Omega;\gamma)})$ is complete, let $(v_{n})_{n\in\mathbb{N}}$ be a Cauchy sequence with respect to the seminorm $\|\cdot\|_{\mathcal{V}(\Omega;\gamma)}$. Then it remains to show that $(v_{n})_{n\in\mathbb{N}}$ converges to a function $v\in\mathcal{V}(\Omega;\gamma)$ with respect to $\|\cdot\|_{\mathcal{V}(\Omega;\gamma)}$. Since $(v_{n})_{n\in\mathbb{N}}$ is Cauchy, we have
    \begin{align*}
        &\lim\limits_{k,\ell\to\infty}\int_{\Omega}(v_{k}(x)-v_{\ell}(x))^{2}\,\mathrm{d}x=0,\\
        \text{and}\quad  &\lim\limits_{k,\ell\to\infty}\int_{\Omega}\int_{\Rd}(v_{k}(x)-v_{k}(y)-(v_{\ell}(x)-v_{\ell}(y)))^{2}\gamma(y,x)\,\mathrm{d}y\,\mathrm{d}x=0.
    \end{align*}
    Hence, due to the completeness  of $\mathrm{L}^{2}(\Omega)$ and $\mathrm{L}^{2}(\Omega\times\Rd)$, there is a $v\in\mathrm{L}^{2}(\Omega)$ and $w\in \mathrm{L}^{2}(\Rd\times\Omega)$ with 
    \begin{align*}
        &\lim\limits_{k\to\infty}\int_{\Omega}(v_{k}(x)-v(x))^{2}\,\mathrm{d}x=0,\\
        \text{and}\quad  &\lim\limits_{k\to\infty}\int_{\Omega}\int_{\Rd}((v_{k}(x)-v_{k}(y))\sqrt{\gamma(y,x)}-w(y,x))^{2}\,\mathrm{d}y\,\mathrm{d}x=0.
    \end{align*}
     We now choose a subsequence $(v_{n_{\ell}})_{\ell\in\mathbb{N}}$, such that we have
     \begin{align*}
        &\lim\limits_{\ell\to \infty}v_{n_{\ell}}(x)=v(x)\quad&&\text{for a.e.\;}x\in\Omega,&\\
        \text{and}\quad&\lim\limits_{\ell\to \infty}(v_{n_{\ell}}(x)-v_{n_{\ell}}(y))\sqrt{\gamma(y,x)}=w(y,x)\quad &&\text{for a.e.\;}(y,x)\in\Rd\times\Omega.&
    \end{align*}
    Without loss of generality we now assume that $\Omega\neq\Rd$. In view of \eqref{eq:DefinitionGamma}
    we set $v(y)=0$ for $y\in\Rd\setminus(\Omega\cup\Gamma)$. For a.e.\;$y\in \Gamma$ we choose a measurable subset
    \begin{equation*}
        \Omega_{y}\subset\{x\in\Omega \text{ such that }\gamma(y,x)>0\}\subset\Rd\quad \text{with}\quad 0<\lambda(\Omega_{y})<\infty.
    \end{equation*}
    The existence of such subsets is given by the definition of $\Gamma$ and the $\sigma$-finiteness of the Lebesgue measure on $\Rd$. We set
    \begin{align*}
        v(y)=\frac{1}{\lambda(\Omega_{y})}\int_{\Omega_{y}}v(x)-\frac{w(y,x)}{\sqrt{\gamma(y,x)}}\,\mathrm{d}x\text{ for a.e.\;}y\in \Gamma.
    \end{align*}
    Because for any $n\in\mathbb{N}$, a.e.\;$y\in\Gamma$ and a.e.\;$x\in\Omega_{y}$ we have
    \begin{align*}
      \vert v_{n}(x)-\frac{(v_{n}(x)-v_{n}(y))\sqrt{\gamma(y,x)}}{\sqrt{\gamma(y,x)}} \vert = \vert v_{n}(y)\vert<\infty.
    \end{align*}
    Thus, by Lebesgue's Dominated Convergence Theorem we see that, for a.e.\;$y\in \Gamma$,
    \begin{align*}
        v(y)=\frac{1}{\lambda(\Omega_{y})}\int_{\Omega_{y}}v(x)-\frac{w(y,x)}{\sqrt{\gamma(y,x)}}\,\mathrm{d}x=\lim\limits_{\ell\to\infty}\frac{1}{\lambda(\Omega_{y})}\int_{\Omega_{y}}v_{n_{\ell}}(y)\,\mathrm{d}x=\lim\limits_{\ell\to\infty}v_{n_{\ell}}(y).
    \end{align*}
    Recalling that $y\in\Rd\setminus(\Omega\cup\Gamma)$ implies $\gamma(y,x)=0$ for a.e.\;$x\in \Omega$, we conclude that for a.e.\;$(y,x)\in\Rd\times \Omega$ we have
    \begin{align*}
        w(y,x)=\lim\limits_{\ell\to \infty}(v_{n_{\ell}}(x)-v_{n_{\ell}}(y))\sqrt{\gamma(y,x)}=(v(x)-v(y))\sqrt{\gamma(y,x)}.
    \end{align*}
    Now we have found a function $v\in \mathcal{V}(\Omega;\gamma)$, for which  $\lim\limits_{\ell\to\infty}\|v_{n_{\ell}}-v\|_{\mathcal{V}(\Omega;\gamma)}=0$ holds and because $(v_{n})_{n\in\mathbb{N}}$ is Cauchy with respect to $\|\cdot\|_{\V}$, we obtain $\lim\limits_{n\to\infty}\|v_{n}-v\|_{\mathcal{V}(\Omega;\gamma)}=0$ as well. Therefore $(\mathcal{V}(\Omega;\gamma),\|\cdot\|_{\mathcal{V}(\Omega;\gamma)})$ is complete.
\end{proof}

In the same manner as Lebesgue spaces are treated, we now identify the elements of $\mathcal{V}(\Omega;\gamma)$ with their respective equivalent class. For this reason we define
\begin{align*}
    \mathrm{N}:=\{v\in \mathcal{V}(\Omega;\gamma)\colon v=0 \text{ a.e.\;on }\Omega\cup\Gamma\}.
\end{align*}
By definition of $\|\cdot\|_{\mathcal{V}(\Omega;\gamma)}$ we have $u\in \operatorname{ker}(\|\cdot\|_{\mathcal{V}(\Omega;\gamma)})$ if and only if  $u=0$ a.e.\;in $\Omega$ and
\begin{align*}
   \int_{\Omega}\int_{\Gamma}u^{2}(y)\gamma(y,x)\,\mathrm{d}y\,\mathrm{d}x=0,
\end{align*}
so that $\mathrm{N}=\operatorname{ker}(\|\cdot\|_{\mathcal{V}(\Omega;\gamma)})$. Now, by defining
\begin{equation*}
    [u]:=u+\mathrm{N}:=\{v\in\mathcal{V}(\Omega;\gamma)\colon v=u\text{ a.e.\;on }\Omega\cup\Gamma\}\quad\text{for }u\in\mathcal{V}(\Omega;\gamma)
\end{equation*}
the corresponding quotient space is given by
\begin{equation*}
    \V:=\mathcal{V}(\Omega;\gamma)/\mathrm{N}=\{[u]\colon u\in\mathcal{V}(\Omega;\gamma)\}.
\end{equation*}
For any $ u,v\in\mathcal{V}(\Omega;\gamma)$ we then have $\langle u_{1},v_{1}\rangle_{\mathcal{V}(\Omega;\gamma)}= \langle u_{2},v_{2}\rangle_{\mathcal{V}(\Omega;\gamma)}$ and $\mathfrak{B}(u_{1},v_{1})=\mathfrak{B}(u_{2},v_{2})$ for all $u_{1},u_{2}\in [u]$ and $v_{1},v_{2}\in [v]$. This implies that, that both mappings $\langle \cdot,\cdot\rangle_{\V}\colon \V\times\V\to\mathbb{R} $ and $\mathfrak{B}\colon \V\times\V\to\mathbb{R} $ defined by 
\begin{align*}
    \langle [u],[v]\rangle_{\V} :=\langle u,v\rangle_{\mathcal{V}( \Omega;\gamma)} \quad\text{and}\quad \mathfrak{B}([u],[v]):=\mathfrak{B}(u,v) \text{ for }u\in[u],v\in[v]
\end{align*}
are well-defined. And for $v\in\mathcal{V}(\Omega;\gamma)$ we have $\langle [v],[v]\rangle_{\V} =0$ if and only if $[v]=\mathrm{N}$. 

Although the elements of $\V$ are in fact equivalent classes and not functions, we consider the elements of $\V$ to be functions which are considered identical, if they are representatives of the same equivalent class. To be more precise, in our view the elements of $\V$ are functions which are defined a.e.\;on $\Omega\cup\Gamma$. 

Thanks to \Cref{thm:comp}, we therefore conclude:
\begin{corollary}\label{corollary:hilb}
    Let $\Omega\subset \Rd$ be nonempty and open and $\gamma\in\kernel$. Then $\V$ is a separable Hilbert space with the inner product $\langle \cdot,\cdot\rangle_{\V}$. In particular, for every $D\in(\V)'$ there exists a unique $u\in\V$ such that for all $v\in\V$ we have
    \begin{equation*}
        D(v)=\int_{\Omega}u(x)v(x)\,\mathrm{d}x+\mathfrak{B}(u,v).
    \end{equation*}
\end{corollary}
\begin{proof}
    Because of \Cref{thm:comp}, the definition of $\V$, \cref{equ:equiv} and the Riesz representation theorem, it solely remains to show that $\V$ is separable. The mapping
    \begin{equation*}
        \mathcal{I}\colon \V\to \mathrm{L}^{2}(\Omega)\times \mathrm{L}^{2}(\Omega\times \Rd),\,u\mapsto \left(u,(u(x)-u(y))\sqrt{\gamma(y,x)}\right)
    \end{equation*}
    is isometric due the definition of the norm in $\V$. Due to the completeness of $\V$, we obtain that $\mathcal{I}(\V)$ is a closed subspace of $ \mathrm{L}^{2}(\Omega)\times \mathrm{L}^{2}(\Omega\times \Rd)$. This product is separable which implies the separability of $\V$.
\end{proof}

\begin{remark} 
    We note that $\gamma\in\kernel$ is not required to be symmetric in this subsection. In fact, also for a nonsymmetric kernel $\gamma\in\kernel$ we have that $ \mathcal{V}(\Omega;\gamma)$, $\langle \cdot,\cdot\rangle_{\mathcal{V}(\Omega;\gamma)}$ and $\mathfrak{B}$ are all well-defined and $\mathcal{V}(\Omega;\gamma)$ is a complete seminormed vector space with respect to $\|\cdot\|_{\mathcal{V}(\Omega;\gamma)}$. And therefore $\V$ is a separable Hilbert space also in this case.
\end{remark}

%%================================%%

\subsection{Nonlocal integration by parts}\label{ssec:gaus}

As mentioned earlier, we now show that problem \eqref{NP} has a variational structure.
\begin{theorem}[Nonlocal integration by parts formula]\label{thm:gauss}
    Let $\Omega\subset \Rd$ be open and bounded and $\gamma\in\kernel$. For a given $k\in\kernel$ with $k>0$ and 
    \begin{equation*}
        \left\|\int_{\Rd}k(y,\cdot)\,\mathrm{d}y\right\|_{\mathrm{L}^{\infty}(\Omega)}<\infty,
    \end{equation*}
    set $\widetilde{\gamma}(y,x):=\max\Bigg\{\gamma(y,x),\dfrac{(\gamma(x,y)-\gamma(y,x))^{2}}{k(y,x)}\Bigg\}$. If $u\in\mathrm{V}(\Omega;\widetilde{\gamma})$ satisfies
    \begin{equation*}
        \int_{\Omega}\left(\int_{\Rd}\vert u(x)\gamma(x,y)-u(y)\gamma(y,x)\vert\,\mathrm{d}y\right)^{2}\,\mathrm{d}x<\infty,
    \end{equation*}
    then for all $v\in\mathrm{V}(\Omega;\widetilde{\gamma})$ we have
\begin{equation}\label{eq:nonlocalPI}
    \begin{aligned}
	    &\int_{\Omega}\La u(x)\,v(x)\,\mathrm{d}x+\int_{\widehat{\Gamma}}\Neu u(y)\,v(y)\,\mathrm{d}y\\
	    =&\frac{1}{2}\int_{\Omega}\int_{\Omega}(u(x)\gamma(x,y)-u(y)\gamma(y,x))(v(x)-v(y))\,\mathrm{d}y\,\mathrm{d}x\\
	    &+\int_{\Omega}\int_{\widehat{\Gamma}}(u(x)\gamma(x,y)-u(y)\gamma(y,x))(v(x)-v(y))\,\mathrm{d}y\,\mathrm{d}x.
    \end{aligned}
\end{equation}
    In particular, we obtain
    \begin{equation} \label{eq:nonlocalPI-specialCase}
        \int_{\Omega}-\La u(x)\,\mathrm{d}x=\int_{\widehat{\Gamma}}\Neu u(y)\,\mathrm{d}y.
    \end{equation}
\end{theorem}
\begin{proof}
    Our desired result will be obtained via Fubini's Theorem. Note, that by our assumptions we have $\La u\in\Ls$ and by applying Hölder's inequality we get for the first term in \eqref{eq:nonlocalPI} that $\int_{\Omega}\La u(x)\,v(x)\,\mathrm{d}x\leqslant \|\La u\|_{\Ls}\|v\|_{\Ls}<\infty$. For the second term in \eqref{eq:nonlocalPI}, the triangle inequality yields
    \begin{align*}
        &\int_{\widehat{\Gamma}}\int_{\Omega}\vert(u(y)\gamma(y,x)-u(x)\gamma(x,y))v(y)\vert\,\mathrm{d}x\,\mathrm{d}y\\
        =&\int_{\widehat{\Gamma}}\int_{\Omega}\vert(u(y)\gamma(y,x)-u(x)\gamma(x,y))(v(x)-v(x)-v(y))\vert\,\mathrm{d}x\,\mathrm{d}y\\
        \leqslant &\int_{\Rd}\int_{\Omega}\vert(u(y)\gamma(y,x)-u(x)\gamma(x,y))v(x)\vert\,\mathrm{d}x\,\mathrm{d}y\\
        &+\int_{\Rd}\int_{\Omega}\vert(u(y)\gamma(y,x)-u(x)\gamma(x,y))(v(x)-v(y))\vert\,\mathrm{d}x\,\mathrm{d}y.
    \end{align*}
    So it remains to show, that the integral
    \begin{equation*}
        \int_{\Omega}\int_{\Rd}\vert(u(x)\gamma(x,y)-u(y)\gamma(y,x))(v(x)-v(y))\vert\,\mathrm{d}y\,\mathrm{d}x
    \end{equation*}
    is finite. By using the triangle inequality again we obtain
    \begin{align*}
        &\int_{\Omega}\int_{\Rd}\vert(u(x)\gamma(x,y)-u(y)\gamma(y,x))(v(x)-v(y))\vert\,\mathrm{d}y\,\mathrm{d}x\\
        =&\int_{\Omega}\int_{\Rd}\vert(u(x)(\gamma(x,y)-\gamma(y,x))+(u(x)-u(y))\gamma(y,x))(v(x)-v(y))\vert\,\mathrm{d}y\,\mathrm{d}x\\
        \leqslant&\int_{\Omega}\int_{\Rd}\vert u(x)(\gamma(x,y)-\gamma(y,x))(v(x)-v(y))\vert\,\mathrm{d}y\,\mathrm{d}x\\
        &+\int_{\Omega}\int_{\Rd}\vert(u(x)-u(y))\gamma(y,x)(v(x)-v(y))\vert\,\mathrm{d}y\,\mathrm{d}x.
    \end{align*}
    While for the second term Hölder's inequality yields
    \begin{equation*}
        \int_{\Omega}\int_{\Rd}\vert(u(x)-u(y))\gamma(y,x)(v(x)-v(y))\vert\,\mathrm{d}y\,\mathrm{d}x\leqslant \|u\|_{\mathrm{V}(\Omega;\widetilde{\gamma})}\|v\|_{\mathrm{V}(\Omega;\widetilde{\gamma})}<\infty
    \end{equation*}
    we estimate the first term by
    \begin{align*}
        &\int_{\Omega}\int_{\Rd}\vert u(x)(\gamma(x,y)-\gamma(y,x))(v(x)-v(y))\vert\,\mathrm{d}y\,\mathrm{d}x\\
        =&\int_{\Omega}\int_{\Rd} \vert u(x)\frac{(\gamma(x,y)-\gamma(y,x))}{\sqrt{k(y,x)}}\sqrt{k(y,x)}(v(x)-v(y))\vert\,\mathrm{d}y\,\mathrm{d}x\\
        \leqslant & \left\|\int_{\Rd}k(y,\cdot)\,\mathrm{d}y\right\|_{\mathrm{L}^{\infty}(\Omega)}^{\frac{1}{2}}\|u\|_{\mathrm{V}(\Omega;\widetilde{\gamma})}\|v\|_{\mathrm{V}(\Omega;\widetilde{\gamma})}\\
        <&\infty.
    \end{align*}
    Therefore, all integrals are, in the Lebesgue-sense, well defined and finite. And the linearity of an integral yields 
    \begin{align*}
        &\int_{\Omega}\int_{\Rd}(u(x)\gamma(x,y)-u(y)\gamma(y,x))v(x)\,\mathrm{d}y\,\mathrm{d}x\\
        =&\int_{\Omega}\int_{\Omega}(u(x)\gamma(x,y)-u(y)\gamma(y,x))v(x)\,\mathrm{d}y\,\mathrm{d}x\\
        &+\int_{\Omega}\int_{\Rd\setminus \Omega}(u(x)\gamma(x,y)-u(y)\gamma(y,x))v(x)\,\mathrm{d}y\,\mathrm{d}x.
    \end{align*}
    Note that $\gamma(x,y)=\gamma(y,x)=0$ holds for a.e.\;$(x,y)\in\Omega\times\Rd\setminus (\Omega\cup\widehat{\Gamma})$ and thereby
    \begin{align*}
        &\int_{\Omega}\int_{\Rd\setminus \Omega}(u(x)\gamma(x,y)-u(y)\gamma(y,x))v(x)\,\mathrm{d}y\,\mathrm{d}x\\
        =&\int_{\Omega}\int_{\widehat{\Gamma}}(u(x)\gamma(x,y)-u(y)\gamma(y,x))v(x)\,\mathrm{d}y\,\mathrm{d}x.
    \end{align*}
    By Fubini's Theorem we conclude
    \begin{align*}
        &\int_{\Omega}\int_{\Omega}(u(x)\gamma(x,y)-u(y)\gamma(y,x))v(x)\,\mathrm{d}y\,\mathrm{d}x\\
        =&\frac{1}{2}\int_{\Omega}\int_{\Omega}(u(x)\gamma(x,y)-u(y)\gamma(y,x))v(x)\,\mathrm{d}y\,\mathrm{d}x\\
        &-\frac{1}{2}\int_{\Omega}\int_{\Omega}(u(y)\gamma(y,x)-u(x)\gamma(x,y))v(x)\,\mathrm{d}y\,\mathrm{d}x\\
        =&\frac{1}{2}\int_{\Omega}\int_{\Omega}(u(x)\gamma(x,y)-u(y)\gamma(y,x))(v(x)-v(y))\,\mathrm{d}y\,\mathrm{d}x
    \end{align*}
    and 
    \begin{align*}
        &\int_{\Omega}\int_{\Rd\setminus \Omega}(u(x)\gamma(x,y)-u(y)\gamma(y,x))v(x)\,\mathrm{d}y\,\mathrm{d}x\\
        =&\int_{\Omega}\int_{\widehat{\Gamma}}(u(x)\gamma(x,y)-u(y)\gamma(y,x))(v(x)-v(y))\,\mathrm{d}y\,\mathrm{d}x\\
        &-\int_{\widehat{\Gamma}}\int_{\Omega}(u(y)\gamma(y,x)-u(x)\gamma(x,y))\mathrm{d}x\,v(y)\,\mathrm{d}y.
    \end{align*}
    Due to $\chi_{\Rd}\in\V$, the special case \eqref{eq:nonlocalPI-specialCase} follows.
\end{proof}

%%================================%%

\section{The nonlocal Neumann problem}\label{ssub:sym}

For the remainder of this section we assume, if not stated otherwise, $\gamma\in\kernel$ to be symmetric.

Let $u\in\V$ be a solution of problem \eqref{NP} which satisfies
\begin{align*}
    \int_{\Omega}\left(\int_{\Rd}\vert u(x)-u(y)\vert\gamma(y,x)\,\mathrm{d}y\right)^{2}\,\mathrm{d}x<\infty.
\end{align*}
Then by the nonlocal integration by parts formula (see \Cref{thm:gauss}) we get, for all $v\in \V$,
\begin{align*}
	\int_{\Omega}{\color{blue}f(x)}v(x)\,\mathrm{d}x&=\int_{\Omega}{\color{blue}\La u(x)}v(x)\,\mathrm{d}x\\
	&=\mathfrak{B}(u,v)-\int_{\Gamma}{\color{red}\Neu u(y)}\,v(y)\,\mathrm{d}y=\mathfrak{B}(u,v)-\int_{\Gamma}{\color{red}g(y)}\,v(y)\,\mathrm{d}y.
\end{align*}
Hence we define our weak solution accordingly.
\begin{definition}\label{de:weak_s}
    Let $\Omega\subset \Rd$ be nonempty and open and $\gamma\in\kernel$ be symmetric. For given measurable functions $f\colon \Omega\to\mathbb{R}$ and $g\colon \Gamma\to\mathbb {R}$, we say that the function $u\in \V$ is a \emph{weak solution} of the Neumann problem \eqref{NP}, if
    \begin{align*}
    	\int_{\Omega}f(x)v(x)\,\mathrm{d}x+\int_{\Gamma}g(y)\,v(y)\,\mathrm{d}y=\mathfrak{B}(u,v)\quad \text{holds for all }v \in \V.
    \end{align*}
    We call the Neumann problem \emph{homogeneous} if $g=0$.
\end{definition}

First, we highlight that the symmetry assumption on $\gamma \in\kernel$ can be relaxed.
\begin{remark}
    Consider the nonlocal problem
    \begin{equation}\tag{*}\label{relaxed}
        \begin{cases}
            \mathcal{L}_{\eta}u(x)=\displaystyle{\int_{\Rd} (u(x)-u(y))\eta(y,x)\,\mathrm{d}y}&=f(x)\quad\text{for }x\in\Omega,\\
            \Neu_{\eta}u(y)=\displaystyle{\int_{\Omega} (u(y)-u(x))\eta(y,x)\,\mathrm{d}x}&=g(y)\quad\text{for } y\in\Gamma(\Omega,\eta),
        \end{cases}
    \end{equation}
    where the measurable function $\eta\colon \Rd\times\Omega\to[0,\infty)$ is only assumed to be symmetric a.e.\;on $\Omega\times\Omega$. By setting 
    \begin{equation*}
        \gamma\colon \Rd\times\Rd\to[0,\infty),\;\gamma(y,x)=\begin{cases}
            \eta(y,x)\quad\text{for }(y,x)\in\Omega\times\Omega,\\
            \eta(y,x)\quad\text{for }(y,x)\in\Rd\setminus \Omega\times\Omega,\\
            \eta(x,y)\quad\text{for }(y,x)\in\Omega\times\Rd\setminus\Omega,
        \end{cases}
    \end{equation*}
    we obtain a well-defined measurable and nonnegative function, i.e., $\gamma\in\kernel$ which is symmetric a.e.\;on $\Rd\times\Rd$ with $\Gamma(\Omega,\gamma)=\Gamma(\Omega,\eta)$ and
    \begin{equation*}
       \mathfrak{B}_{\eta}(u,v)=\mathfrak{B}_{\gamma}(u,v)\;\text{ for all } u,v\in\V=\mathrm{V}(\Omega;\eta).
    \end{equation*}
    Moreover, % we have defined a symmetric kernel $\gamma\in\kernel$ such that 
    $u\colon\Rd\to\mathbb{R}$ is a strong solution of problem \eqref{relaxed}  if and only if $u\colon\Rd\to\mathbb{R}$ is a strong solution of problem $\eqref{NP}$. Therefore, we define the weak solution of problem \eqref{relaxed} accordingly:
    
    {\it
    The function $u\in \mathrm{V}(\Omega;\eta)$ is a weak solution of the problem \eqref{relaxed} if
    \begin{align*}
        \int_{\Omega}f(x)v(x)\,\mathrm{d}x-\int_{\Gamma(\Omega,\eta)}g(y)\,v(y)\,\mathrm{d}y=\mathfrak{B}_{\eta}(u,v)\quad \text{holds for all }v \in \mathcal{V}(\Omega;\eta).
    \end{align*}}
\end{remark}

In the next remark we study under which assumptions the weak and strong solutions of the homogeneous Neumann problem are equivalent.
\begin{remark}
    Let $f\in\Ls$ and let $u\in\V$ be a weak solution of the homogeneous Neumann problem \eqref{NP} satisfying 
    \begin{align*}
        \int_{\Omega}\left(\int_{\Rd}\vert u(x)-u(y)\vert\gamma(y,x)\,\mathrm{d}y\right)^{2}\,\mathrm{d}x<\infty.
    \end{align*}
    By  the definition of the weak solution (\Cref{de:weak_s}) and the nonlocal integration by parts formula (\Cref{thm:gauss}) we find
    \begin{align*}
        \int_{\Omega}\mathcal{L}u(x)v(x)\,\mathrm{d}x+\int_{\Gamma}\Neu u(y)v(y)\,\mathrm{d}y=\int_{\Omega}f(x)v(x)\,\mathrm{d}x\text{ for all }v\in\V.
    \end{align*}
    If every infinitely differentiable function with compact support  is an element of $\V$ (i.e. $\mathrm{C}^{\infty}_{0}(\Rd)\subset\V$), then also
    \begin{align*}
        \int_{\Omega}\mathcal{L}u(x)v(x)\,\mathrm{d}x+\int_{\Gamma}\Neu u(y)v(y)\,\mathrm{d}y=\int_{\Omega}f(x)v(x)\,\mathrm{d}x\text{ for all }v\in\mathrm{C}^{\infty}_{0}(\Rd),
    \end{align*}
    so that by \cite[Corollary 4.24.]{brezis2010functional} we have $\mathcal{L}u=f$ a.e.\;on $\Omega$ and $\Neu u=0$ a.e.\;on $\Gamma$. 
    
    If for a.e.\;$x\in\Omega\cup\Gamma$ there is an $R>0$ such that for all $0<r<R$ we have $\chi_{\mathrm{B}_{r}(x)}\in\V$, then by \cite[Theorem 7.7]{rudin1987real} we get
    \begin{equation*}
        \La u(x_{0})=\lim\limits_{r\to 0}\int_{\Omega}\La u(x)\frac{\chi_{\mathrm{B}_{r}(x_{0})}(x)}{\lambda(\mathrm{B}_{r}(0))}\,\mathrm{d}x=\lim\limits_{r\to 0}\int_{\Omega}f(x)\frac{\chi_{\mathrm{B}_{r}(x_{0})}(x)}{\lambda(\mathrm{B}_{r}(0))}\,\mathrm{d}x=f(x_{0})
    \end{equation*}
    for a.e.\;$x_{0}\in\Omega$ and 
    \begin{equation*}
        \Neu u(y_{0})=\lim\limits_{r\to 0}\int_{\Gamma}\Neu u(y)\chi_{\mathrm{B}_{r}(y_{0})}(y)\,\mathrm{d}y=0
    \end{equation*}
    for a.e.\;$y_{0}\in\Gamma$.
\end{remark}

Before we present an existence result for the weak solution of problem \eqref{NP}, we show that \Cref{corollary:hilb} brings in the existence of a weak solution for the following regularized variant of \eqref{NP} with homogeneous Neumann constraints.
\begin{remark}\label{rem:trivial_neum_case}
    Let us consider the problem
    \begin{equation*}
        \left\{
            \begin{array}{llr}
                \mathcal{L}u(x)+\kappa(x) u(x)&=f(x)\quad&\text{for }x\in\Omega,\\
                \Neu u(y)&=0\quad&\text{for } y\in\Gamma,
            \end{array}
        \right.
    \end{equation*}
    for $\Omega\subset \Rd$ bounded and open, $f\in\Ls $, a measurable function $\kappa\colon\Omega\to[\alpha,\beta]$ with $0<\alpha<\beta<\infty$ and a symmetric function $\gamma\in\kernel$. Because  
    \begin{equation*}
        \min\{\frac{1}{2},\alpha\}\|u\|_{\V}^{2}\leqslant\mathfrak{B}(u,u)+\int_{\Omega}\kappa(x)u^{2}(x)\,\mathrm{d}x\leqslant \max\{1,\beta\} \|u\|_{\V}^{2}
    \end{equation*}
    holds for all $u\in\V$, the existence of a unique weak solution of this problem, i.e., the function $u\in\V$ solving
    \begin{equation*}
        \int_{\Omega}f(x)v(x)\,\mathrm{d}x=\mathfrak{B}(u,v)+\int_{\Omega}\kappa(x)u(x)v(x)\,\mathrm{d}x\quad \text{for all }v \in \V,
    \end{equation*}
    is given by \Cref{corollary:hilb}. The case $\kappa=0$ however requires special treatment, since neither the existence nor the uniqueness is guaranteed. Note that all constant functions belong to $\V$ and that we have $\mathfrak{B}(u,v)=\mathfrak{B}(u+c,v)$ for all $u,v\in\V$ and constants $c\in\mathbb{R}$. Therefore, the weak solution cannot be unique in $\V$. And the existence of a weak solution to our Neumann problem already implies
    \begin{align*}
        \int_{\Omega}f(x)\,\mathrm{d}x+\int_{\Gamma}g(y)\,\mathrm{d}y=0.
    \end{align*}
    In other words a compatibility condition is necessary for the existence.
\end{remark}

Using a similar approach as is used for the local Neumann problem that can be found in Brenner and Scott \cite[Section 2.5.]{brenner2007mathematical}), we now present an existence and uniqueness result for the weak solution of the homogeneous Neumann problem \eqref{NP} (i.e., $\kappa=0$ in view of \Cref{rem:trivial_neum_case}).
\begin{theorem}	\label{thm:exriesz}
	Let $\Omega\subset\Rd$ be bounded, nonempty and open, $f\in\mathrm{L}^{2}(\Omega)$ and $\gamma\in\kernel$ be symmetric. Then the homogeneous Neumann problem \eqref{NP}, i.e.,
	\begin{equation*}
    \left\{
    \begin{array}{lll}
         \mathcal{L}u&=f\quad&\text{on }\Omega,\\
         \Neu u&=0\quad&\text{on }\Gamma,
    \end{array}
    \right.
    \end{equation*}
    has a weak solution if 
    \begin{align}
    	\int_{\Omega}f(x)\,\mathrm{d}x=0\tag{Compatibility condition}
    \end{align}
    and if there is a constant $C>0$, such that for all $v\in \V$ we have
    \begin{align}
    	\int_{\Omega}\int_{\Omega}(v(x)-v(y))^{2}\,\mathrm{d}y\,\mathrm{d}x \leqslant C\int_{\Omega}\int_{\Rd}(v(x)-v(y))^{2}\gamma(y,x)\,\mathrm{d}y\,\mathrm{d}x.\tag{\text{Poincar\'{e} inequality}}
    \end{align}
    Furthermore the weak solution is unique up to an additive constant.
\end{theorem}
\begin{proof}
	First we remark that for all $u\in \V$ we have
	\begin{align*}
	    2\left(\lambda(\Omega)\int_{\Omega}u^{2}(x)\,\mathrm{d}x-\left(\int_{\Omega}u(x)\,\mathrm{d}x\right)^{2}\right)&=\int_{\Omega}\int_{\Omega}(u(x)-u(y))^{2}\,\mathrm{d}y\,\mathrm{d}x\\
    	&\leqslant  C\int_{\Omega}\int_{\Rd}(u(x)-u(y))^{2}\gamma(y,x)\,\mathrm{d}y\,\mathrm{d}x
    \end{align*}
    and we introduce the space 
    \begin{align*}
    	\widehat{V}(\Omega;\gamma):=\left\{v\in \V \colon\int_{\Omega}v(x)\,\mathrm{d}x=0\right\}.
    \end{align*}
    Let us define the average
    \begin{align*}
    	u_{\Omega}:=\frac{1}{\lambda(\Omega)}\int_{\Omega}u\,\mathrm{d}x
    \end{align*}
    for $u\in\V$. Since $u-u_{\Omega}$ is for all $u\in \V$ an element of $\widehat{V}(\Omega;\gamma)$ we have
    \begin{align*}
	    \int_{\Omega}(u(x)-u_{\Omega})^{2}\,\mathrm{d}x\leqslant \frac{C}{2\lambda(\Omega)}\int_{\Omega}\int_{\Rd}(u(x)-u(y))^{2}\gamma(y,x)\,\mathrm{d}y\,\mathrm{d}x\leqslant\frac{C}{\lambda(\Omega)} \, \mathfrak{B}(u,u).
    \end{align*}
    For all $u\in \widehat{V}(\Omega;\gamma)$ we have due to the Poincar\'e inequality
    \begin{align}\label{eqnorm}
	    \min\left\{\dfrac{1}{4},\dfrac{\lambda(\Omega)}{2\,C}\right\} \|u\|_{\V}^{2}\leqslant &\dfrac{\lambda(\Omega)}{2\,C}\int_{\Omega}u^{2}(x)\,\mathrm{d}x+ \dfrac{1}{4}\int_{\Omega}\int_{\Rd}(u(x)-u(y))^{2}\gamma(y,x)\,\mathrm{d}y\,\mathrm{d}x\\
    	\leqslant &\mathfrak{B}(u,u)\\
	    \leqslant &\left(\int_{\Omega}u^{2}(x)\,\mathrm{d}x+ \int_{\Omega}\int_{\Rd}(u(x)-u(y))^{2}\gamma(y,x)\,\mathrm{d}y\,\mathrm{d}x\right)\\
	    =& \|u\|_{\V}^{2}.
    \end{align}
   We obtain that $\mathfrak{B}(w,w)=0$ for $w \in \widehat{V}(\Omega;\gamma)$ implies $\|w\|_{\V}=0$, i.e., $w=0$ a.e.\;on $\Omega\cup\Gamma$. Thereby $\mathfrak{B}$ is an inner product on the space $ \widehat{V}(\Omega;\gamma)$. Because $\Omega$ is bounded and because \eqref{eqnorm}, we see that $\widehat{V}(\Omega;\gamma)$ is a Hilbert space with respect to the inner product $\mathfrak{B}$. Let us consider the linear functional
    \begin{align*}
    	\Lambda (v)=\int_{\Omega}f(x)v(x)\,\mathrm{d}x\quad \text{on }\widehat{V}(\Omega;\gamma),
    \end{align*}
    which is bounded by the Cauchy–Schwarz inequality $\Lambda$. Therefore, by the Riesz representation theorem there is a unique $u\in \widehat{V}(\Omega;\gamma)$ such that 
    \begin{align*}
    	\Lambda (v)=\mathfrak{B}(u,v)\quad\text{for all }v\in\widehat{V}(\Omega;\gamma).
    \end{align*}
    Because $\int_{\Omega}f(x)\,\mathrm{d}x=0$ holds we get
    \begin{align*}
    	\int_{\Omega}f(x)v(x)\,\mathrm{d}x=&\int_{\Omega}f(x)(v(x)-v_{\Omega})\,\mathrm{d}x\\
    	=&\mathfrak{B}(u,v-v_{\Omega})=\mathfrak{B}(u,v)\text{ for all } v\in \V.
    \end{align*}
    It remains to show that weak solutions in $\V$ are unique up to a constant. For that reason let both $u,v\in\V$ be weak solutions for our nonlocal Neumann problem. This means $u-u_{\Omega}$ and $v-v_{\Omega}$ are weak solutions as well and because the weak solution is unique in $\widehat{V}(\Omega;\gamma)$ we obtain $u=v-v_{\Omega}+u_{\Omega}$.
\end{proof}

Now we study the nonhomogeneous nonlocal Neumann problem. For the existence of a weak solution, the following linear operator
\begin{align*}
    \V\to\mathbb{R},\quad v\mapsto \int_{\Gamma}g(y)v(y)\,\mathrm{d}y
\end{align*}
must exist and be well-defined.
\begin{definition}
    Let $\Omega\subset\Rd$ be bounded, nonempty and open and let $\gamma\in\kernel$.
    \begin{enumerate}[label=(\roman*)]
        \item We say that a measurable function $g\colon \Gamma\to\mathbb{R}$ satisfies the \emph{functional condition (on $\V$)} if 
         \begin{align*}
            v\mapsto\int_{\Gamma}g(y)v(y)\,\mathrm{d}y,
        \end{align*}
        is a linear functional on $\V$. If the functional is continuous we say that $g\colon \Gamma\to\mathbb{R}$ satisfies the \emph{continuous functional condition}.
        \item We say that a measurable function $g\colon \Rd\to\mathbb{R}$ satisfies the \emph{(continuous) functional condition} if $g\vert_{\Gamma}$ satisfies the (continuous) functional condition. 
        %\item We say that a measurable function $g\colon \Gamma\to\mathbb{R}$ satisfies the \emph{r-condition} if for a.e.\;$y\in\Gamma$ with $\int_{\Omega}\gamma(y,x)\,\mathrm{d}x=\infty$ we have $g(y)=0$ and both integrals
        %\begin{align*}
    	%\int_{\Gamma}\frac{g^{2}(y)}{\int_{\Omega}\gamma(z,y)\,\mathrm{d}z}\,\mathrm{d}y\quad \text{ and }\quad \int_{\Omega}\Big(\int_{\Gamma}\frac{\vert g(y)\vert\gamma(y,x)}{\int_{\Omega}\gamma(z,y)\,\mathrm{d}z}\,\mathrm{d}y\Big)^{2}\,\mathrm{d}x
    %	\end{align*}
    %are finite. 
    \end{enumerate}
\end{definition}

Imposing the continuous functional condition on the Neumann data, $g$ is sufficient for the existence of a weak solution to the nonhomogeneous Neumann problem as shown in the next theorem.

%%%%
%%%
%%
%
%{\color{red}In fact, for a given $f\in \mathrm{L}^{2}(\Omega)$, the linear functional
%\begin{align*}
%    \Lambda(v)=\int_{\Omega}f(x)v(x)\,\mathrm{d}x+\int_{\Gamma}g(y)v(y)\,\mathrm{d}y,\quad v\in\V,
%\end{align*}
%is bounded. Using this bounded linear functional we show that a weak solution for the nonhomogeneous Neumann problem exists, by following the proof of \Cref{thm:exriesz}, where the compatibility condition becomes
%\begin{align*}
%	\int_{\Omega}f(x)\,\mathrm{d}x+\int_{\Gamma}g(y)\,\mathrm{d}y=0
%\end{align*}
%and the nonlocal Poincar\'e inequality remains the same:}
%
%%
%%%
%%%%
\begin{theorem}	\label{thm:exnonhom}
	Let $\Omega\subset\Rd$ be bounded, nonempty and open, $f\in\mathrm{L}^{2}(\Omega)$, $\gamma\in\kernel$ be symmetric and $g\colon \Gamma\to \mathbb{R}$ satisfy the continuous functional condition. Then the nonhomogeneous Neumann problem \eqref{NP}, i.e.,
    \begin{equation*}
        \left\{
        \begin{array}{ll}
            \mathcal{L}u&=f\quad\text{on }\Omega,\\
            \Neu u&=g\quad\text{on } \Gamma,
        \end{array}
        \right.
    \end{equation*}
    has a weak solution if 
    \begin{align}
    	\int_{\Omega}f(x)\,\mathrm{d}x+\int_{\Gamma}g(y)\,\mathrm{d}y=0\tag{Compatibility condition}
    \end{align}
    and if there is a constant $C>0$, such that for all $v\in \V$ we have
    \begin{align}
    	\int_{\Omega}\int_{\Omega}(v(x)-v(y))^{2}\,\mathrm{d}y\,\mathrm{d}x \leqslant C\int_{\Omega}\int_{\Rd}(v(x)-v(y))^{2}\gamma(y,x)\,\mathrm{d}y\,\mathrm{d}x.\tag{\text{Poincar\'e  inequality}}
    \end{align}
    Furthermore the weak solution is unique up to an additive constant.
\end{theorem}
\begin{proof}
	As in the proof of \Cref{thm:exriesz} we see that 
    \begin{align*}
    	\widehat{V}(\Omega;\gamma):=\left\{v\in \V \colon\int_{\Omega}v(x)\,\mathrm{d}x=0\right\}
    \end{align*}
     is a Hilbert space with respect to the inner product $\mathfrak{B}$.  Next, we consider the linear functional
    \begin{equation*}
        \Lambda (v)=\int_{\Omega}f(x)v(x)\,\mathrm{d}x+\int_{\Gamma}g(y)v(y)\,\mathrm{d}y\quad \text{for}~v\in\widehat{V}(\Omega;\gamma).
    \end{equation*}
    This functional is bounded, so that by the Riesz representation theorem there is a unique $u\in \widehat{V}(\Omega;\gamma)$ satisfying
    \begin{equation*}
        \Lambda (v)=\mathfrak{B}(u,v)\quad\text{for all }v\in\widehat{V}(\Omega;\gamma).
    \end{equation*}
    Similarly to the proof of \Cref{thm:exriesz}, we find that the compatibility condition implies
     \begin{equation*}
        \Lambda (v)=\mathfrak{B}(u,v)\quad\text{for all }v\in\V.
    \end{equation*}
    Finally, for two weak solutions $u_1,u_2\in\V$ we get that there mean--centered versions are equal, i.e.,
    \begin{align*}
        u_1-\frac{1}{\lambda(\Omega)}\int_{\Omega}u_1(x)\,\mathrm{d}x=u_2-\frac{1}{\lambda(\Omega)}\int_{\Omega}u_2(x)\,\mathrm{d}x.
    \end{align*}
     Thus, weak solutions are unique up to an additive constant.
\end{proof}
In \cref{sec:tr} we present some sufficient assumptions for the continuous functional condition.

An approach to study the local nonhomogeneous Neumann problem is to transform it into an equivalent homogeneous problem by using the linear dependence of the weak solution on the right hand side. 
\begin{remark}
    Let the assumptions of \Cref{thm:exnonhom} be satisfied and $u\in\V$ be a weak solution of the nonhomogeneous Neumann problem \eqref{NP}. Furthermore, let $\widetilde{g}\in\V$ be the weak solution of the Neumann problem
    \begin{equation*}
        \left\{
        \begin{array}{ll}
            \mathcal{L}u+\lambda u&=0\quad\text{on }\Omega,\\
            \Neu u&=g\quad\text{on } \Gamma,
        \end{array}
        \right.
    \end{equation*}
    for a given $\lambda>0$. More precisely, let $\widetilde{g}$ satisfy
    \begin{equation*}
        \mathfrak{B}(\widetilde{g},v)+\lambda\int_{\Omega}\widetilde{g}(x)v(x)\,\mathrm{d}x=\int_{\Gamma}g(y)v(y)\,\mathrm{d}y\quad \text{ for all }v\in\V.
    \end{equation*}
    The existence of $\widetilde{g}$ is given by the Lax--Milgram Theorem. Then we obtain that $\widetilde{u}=u-\widetilde{g}$ is a weak solution of 
    \begin{equation*}
        \left\{
        \begin{array}{lll}
            \mathcal{L}u&=f+\lambda\widetilde{g}\quad&\text{on }\Omega,\\
            \Neu u&=0\quad&\text{on } \Gamma.
        \end{array}
        \right.
    \end{equation*}
\end{remark}

We now show that depending on the nonlocal Neumann boundary condition, there is an explicit way to transform our nonhomogeneous Neumann problem into an equivalent homogeneous problem:
\begin{theorem}\label{thm:refl}
	Let $\Omega\subset\Rd$ be a bounded, nonempty and open, $\gamma\in\kernel$ be symmetric and $g\colon\Gamma\to\mathbb{R}$ such that for a.e.\;$y\in\Gamma$ with $\int_{\Omega}\gamma(y,x)\,\mathrm{d}x=\infty$ we have $g(y)=0$ and such that both integrals
    \begin{align*}
    	\int_{\Gamma}\frac{g^{2}(y)}{\int_{\Omega}\gamma(z,y)\,\mathrm{d}z}\,\mathrm{d}y\quad \text{ and }\quad \int_{\Omega}\Big(\int_{\Gamma}\frac{\vert g(y)\vert\gamma(y,x)}{\int_{\Omega}\gamma(z,y)\,\mathrm{d}z}\,\mathrm{d}y\Big)^{2}\,\mathrm{d}x
    \end{align*}
    are finite.
    \begin{enumerate}[label=(\roman*)]
        \item Then the continuous functional condition is satisfied by $g$ and we have $g\in\mathrm{L}^{1}(\Gamma)$.
        \item Let $\widetilde{g}$ be the zero extension of $\frac{g(\cdot)}{\int_{\Omega}\gamma(\cdot,z)\,\mathrm{d}z}$ outside $\Gamma$, i.e.,
        	\begin{equation*}
        	    \widetilde{g}\colon\Rd\to\mathbb{R},\;\widetilde{g}(y)=\begin{cases}
        	        \frac{g(\cdot)}{\int_{\Omega}\gamma(\cdot,z)\,\mathrm{d}z},\quad &y\in\Gamma,\\
        	        0,\quad&y\in\Rd\setminus\Gamma.
        	    \end{cases}
        	\end{equation*}
        	Then $\widetilde{g}\in\V$ is a strong solution of the nonhomogeneous Neumann problem
            \begin{align*}
                \left\{
                    \begin{array}{llr}
                	    \mathcal{L}u(x)&=-\int_{\Gamma}\frac{g(y)\gamma(y,x)}{\int_{\Omega}\gamma(y,z)\,\mathrm{d}z}\,\mathrm{d}y&\text{for }x\in\Omega,\\
                	    \Neu u(y)&=g(y)&\text{for }y\in\Gamma.
                	\end{array}
                \right.
            \end{align*}
        \item A function $u\in\V$ is a weak solution of the nonhomogeneous Neumann problem \eqref{NP}, i.e.,
    \begin{align*}
        \left\{
        \begin{array}{lll}
    	    \mathcal{L}u(x)&=f(x)\quad\text{for }x\in\Omega,&\\
    	    \Neu u(y)&=g(y)\quad\text{for }y\in\Gamma,
    	\end{array}
        \right.
    \end{align*}
    if and only if $u(\cdot)+\frac{g(\cdot)}{\int_{\Omega}\gamma(z,\cdot)\,\mathrm{d}z}\chi_{\Gamma}(\cdot)\in \V$ is a weak solution of the following homogeneous Neumann problem
    \begin{equation}\label{eq:transHomoProb}
        \begin{aligned}
            \left\{
            \begin{array}{llr}
    		\mathcal{L}\widetilde{u}(x)&=f(x)+\int_{\Gamma}\frac{g(y)\gamma(x,y)}{\int_{\Omega}\gamma(z,y)\,\mathrm{d}z}\,\mathrm{d}y\quad&\text{for }x\in\Omega,\\
    		\Neu\widetilde{u}(y)&=0\quad&\text{for }y\in\Gamma.
    		\end{array}
            \right.
    	\end{aligned}
    \end{equation}
        \end{enumerate}
    
\end{theorem}
\begin{proof}
    For all $v\in\V$ we find
	 \begin{align*}
        &\left\vert\,\int_{\Gamma}g(y)v(y)\,\mathrm{d}y\right\vert\\
        \leqslant&\int_{\Gamma}\left\vert\,g(y)v(y)\right\vert\,\mathrm{d}y\\
        =&\int_{\Omega}\int_{\Gamma}\frac{\vert g(y)\vert\gamma(y,x)}{\int_{\Omega}\gamma(y,z)\,\mathrm{d}z}\vert v(y)\vert\,\mathrm{d}y\,\,\mathrm{d}x\\
        \leqslant  &\int_{\Omega}\int_{\Gamma}\frac{\vert g(y)\vert\gamma(y,x)}{\int_{\Omega}\gamma(y,z)\,\mathrm{d}z}\,\mathrm{d}y\,\vert v(x)\vert\,\mathrm{d}x+\int_{\Omega}\int_{\Gamma}\frac{\vert g(y)\vert}{\int_{\Omega}\gamma(y,z)\,\mathrm{d}z}\vert v(x)-v(y)\vert\gamma(y,x)\,\mathrm{d}y\,\mathrm{d}x\\
        \leqslant & \|v\|_{\V} \max\left\{\int_{\Omega}\Big(\int_{\Gamma}\frac{\vert g(y)\vert\gamma(y,x)}{\int_{\Omega}\gamma(y,z)\,\mathrm{d}z}\,\mathrm{d}y\Big)^{2}\,\mathrm{d}x,\,\int_{\Gamma}\frac{g^{2}(y)}{\int_{\Omega}\gamma(y,z)\,\mathrm{d}z}\,\mathrm{d}y\right\},
    \end{align*}
    where the two terms in the $\max$--function are finite due to our assumptions. Thus $g$ satisfies the continuous functional condition. And by choosing a constant $v$, we see that $g\in \mathrm{L}^{1}(\Gamma)$. Because of our assumptions we further have
    \begin{equation*}
        \|\widetilde{g}\|_{\V}=\int_{\Omega}\int_{\Rd}\widetilde{g}^{2}(y)\gamma(y,x)\,\mathrm{d}y\,\mathrm{d}x=\int_{\Gamma}\frac{g^{2}(y)}{\int_{\Omega}\gamma(z,y)\,\mathrm{d}z}\,\mathrm{d}y<\infty,
    \end{equation*}
    so that $\widetilde{g}\in\V$. The rest follows by evaluating both $\La \widetilde{g}$ and $\Neu \widetilde{g}$ and by the linear dependence of the weak solution on the right hand side.
\end{proof}

\begin{remark}\label{rem:nons}
    In the case that $\gamma\in \kernel$ is nonsymmetric, the variational formulation of problem \eqref{NP} can also be obtained by the nonlocal integration by parts formula \Cref{thm:gauss}. However, unlike the symmetric case, there is, in general, no ``natural'' test function space in the nonsymmetric case. Furthermore, for all $\gamma\in\kernel$ we get
    \begin{align*}
        &\Neu u(y_{1})=-\int_{\Omega}u(x)\gamma(x,y_{1})\,\mathrm{d}x& &\text{for }y_{1}\in\widehat{\Gamma}(\Omega,\gamma)\setminus\Gamma(\Omega,\gamma)\\
        \text{and}\quad &\Neu u(y_{2})=u(y_{2})\int_{\Omega}\gamma(y_{2},x)\,\mathrm{d}x& &\text{for }y_{2}\in\widehat{\Gamma}(\Omega,\gamma)\setminus\Gamma(\Omega,\gamma^{\top}).
    \end{align*}
    So on $\widehat{\Gamma}(\Omega,\gamma)\setminus\Gamma(\Omega,\gamma)$ our Neumann condition reduces to a weighted volume constraint our Neumann condition reduces to a Dirichlet condition. For our nonlocal operator $\La$ we further have by Fubini's theorem
    \begin{equation*}
        \int_{\Omega}\int_{\widehat{\Gamma}(\Omega,\gamma)\setminus\Gamma(\Omega,\gamma)}\gamma(y,x)\,\mathrm{d}y\,\mathrm{d}x=\int_{\widehat{\Gamma}(\Omega,\gamma)\setminus\Gamma(\Omega,\gamma)}\int_{\Omega}\gamma(y,x)\,\mathrm{d}x\,\mathrm{d}y=0
    \end{equation*}
    and therefore, for a.e. $x\in\Omega$, 
    \begin{equation*}
        \begin{split}
            \La u(x)=&\int_{\Omega}u(x)\gamma(x,y)-u(y)\gamma(y,x)\,\mathrm{d}y\\
            &+\int_{\widehat{\Gamma}(\Omega,\gamma)\setminus\Gamma(\Omega,\gamma)}u(x)\gamma(x,y)\,\mathrm{d}y\\
            &+\int_{\Gamma(\Omega,\gamma)}u(x)\gamma(x,y)-u(y)\gamma(y,x)\,\mathrm{d}y.
        \end{split}
    \end{equation*}
\end{remark}

In the following Theorem we discuss some sufficient assumptions on a not necessarily symmetric kernel such that our nonlocal Neumann problem has a weak solution.
\begin{theorem}	\label{thm:nonsym}
	Let $\Omega\subset\Rd$ be bounded, nonempty and open, $f\in\mathrm{L}^{2}(\Omega)$, $\alpha\in\mathrm{L}^{\infty}(\Omega)$ and $\gamma\in\kernel$ with
	\begin{equation*}
	    \Big\|\frac{1}{2}\int_{\Omega}\vert\gamma(\cdot,y)-\gamma(y,\cdot)\vert\,\mathrm{d}y\Big\|_{\mathrm{L}^{\infty}(\Omega)}+\Big\|\int_{\widehat{\Gamma}}\vert\gamma(\cdot,y)-\gamma(y,\cdot)\vert\,\mathrm{d}y\Big\|_{\mathrm{L}^{\infty}(\Omega)}<\infty.
	\end{equation*}
	We assume that there is a constant $C>0$ with $\gamma(x,y)\leqslant C\gamma(y,x)$ for a.e.\;$(y,x)\in\widehat{\Gamma}\times\Omega$ and such that for a.e.\;$x\in\Omega$ we have
	\begin{equation*}
        0<c<\alpha(x)+\frac{1}{2}\int_{\Omega}\gamma(x,y)-\gamma(y,x)\,\mathrm{d}y-\frac{C+1}{2}\int_{\widehat{\Gamma}}\vert\gamma(x,y)-\gamma(y,x)\vert\,\mathrm{d}y.
    \end{equation*}
    Then the Neumann problem
	\begin{equation*}
    \left\{
    \begin{array}{llr}
         \mathcal{L}u(x)+\alpha(x) u(x)&=f(x) &\text{for }x\in\Omega,\\
         \Neu u(y)&=0 &\text{for }y\in\widehat{\Gamma},
    \end{array}
    \right.
    \end{equation*}
    has a weak solution, namely there is a function $u\in\V$ solving
    \begin{equation*}
        \begin{split}
            \widetilde{\mathfrak{B}}(u,v):=&\frac{1}{2}\int_{\Omega}\int_{\Omega}(u(x)\gamma(x,y)-u(y)\gamma(y,x))(v(x)-v(y))\,\mathrm{d}y\,\mathrm{d}x\\
            &+\int_{\Omega}\int_{\widehat{\Gamma}}(u(x)\gamma(x,y)-u(y)\gamma(y,x))(v(x)-v(y))\,\mathrm{d}y\,\mathrm{d}x\\
            & +\int_{\Omega}u(x)v(x)\alpha(x)\,\mathrm{d}x\\
            =&\int_{\Omega}f(x)v(x)\,\mathrm{d}x
        \end{split}
    \end{equation*}
    for all $v\in\V$.
\end{theorem}
\begin{proof}
    By \Cref{thm:gauss}, we see that the definition of the weak solution is justified. By Hölder's inequality there is a constant $K>0$ such that for all $u,v\in\V$ we have  
    \begin{align*}
        \vert\widetilde{\mathfrak{B}}(u,v)\vert
        \leqslant&\int_{\Omega}\int_{\Rd}\vert u(x)\gamma(x,y)-u(y)\gamma(y,x)\vert\vert v(x)-v(y)\vert\,\mathrm{d}y\,\mathrm{d}x\\
        &+\int_{\Omega}\vert u(x)\vert \vert v(x)\vert \alpha(x)\,\mathrm{d}x\\
        \leqslant&\int_{\Omega}\int_{\Rd}\vert u(x)-u(y)\vert \vert v(x)-v(y)\vert \gamma(y,x)\,\mathrm{d}y\,\mathrm{d}x\\
        &+\int_{\Omega}\int_{\Rd} \vert u(x)\vert \vert \gamma(x,y)-\gamma(y,x)\vert\vert v(x)-v(y)\vert\,\mathrm{d}y\,\mathrm{d}x\\
        &+\int_{\Omega}\vert u(x)\vert\vert v(x)\vert\alpha(x)\,\mathrm{d}x\\
        \leqslant & K \|u\|_{\V}\|v\|_{\V}.
    \end{align*}
    Therefore $\widetilde{\mathfrak{B}}$ is bounded on $\V\times\V$ with
    \begin{equation*}
        \begin{split}
            \widetilde{\mathfrak{B}}(u,v)=&\mathfrak{B}(u,v)\\
            &+\frac{1}{2}\int_{\Omega}\int_{\Omega}  u(x)( \gamma(x,y)-\gamma(y,x))(u(x)-u(y))\,\mathrm{d}y\,\mathrm{d}x\\
            &+\int_{\Omega}\int_{\Gamma}  u(x)( \gamma(x,y)-\gamma(y,x))(u(x)-u(y))\,\mathrm{d}y\,\mathrm{d}x\\
            & +\int_{\Omega}u(x)v(x)\alpha(x)\,\mathrm{d}x.
        \end{split}
    \end{equation*}
    And because of 
    \begin{equation*}
        \begin{split}
            \int_{\Omega}\int_{\Omega}  u(x)&( \gamma(x,y)-\gamma(y,x))(v(x)-v(y))\,\mathrm{d}y\,\mathrm{d}x\\
            =&\int_{\Omega}\int_{\Omega}  u(y)( \gamma(x,y)-\gamma(y,x))(v(x)-v(y))\,\mathrm{d}y\,\mathrm{d}x,
        \end{split}
    \end{equation*}
    we obtain, for $u,v\in\V$, 
    \begin{align*}
       \int_{\Omega}\int_{\Omega}  u(x) &( \gamma(x,y)-\gamma(y,x))(u(x)-u(y))\,\mathrm{d}y\,\mathrm{d}x\\
        =&\frac{1}{2}\int_{\Omega}\int_{\Omega}  (u(x)+u(y))( \gamma(x,y)-\gamma(y,x))(u(x)-u(y))\,\mathrm{d}y\,\mathrm{d}x\\
        =&\int_{\Omega}u^{2}(x)\int_{\Omega}  ( \gamma(x,y)-\gamma(y,x))\,\mathrm{d}y\,\mathrm{d}x.
    \end{align*}
    Moreover we estimate 
    \begin{align*}
        &\left\vert\int_{\Omega}\int_{\widehat{\Gamma}}  u(x)( \gamma(x,y)-\gamma(y,x))(u(x)-u(y))\sqrt{\frac{C+1}{C+1}}\,\mathrm{d}y\,\mathrm{d}x\right\vert\\
        \leqslant&\frac{C+1}{2}\int_{\Omega}u^{2}(x)\int_{\widehat{\Gamma}}\vert\gamma(x,y)-\gamma(y,x)\vert\,\mathrm{d}y\,\mathrm{d}x\\
        &+\frac{1}{2}\int_{\Omega}\int_{\widehat{\Gamma}}  (u(x)-u(y))^{2}\frac{\vert\gamma(x,y)-\gamma(y,x)\vert}{C+1}\,\mathrm{d}y\,\mathrm{d}x\\
        \leqslant&\frac{C+1}{2}\int_{\Omega}u^{2}(x)\int_{\widehat{\Gamma}}\vert\gamma(x,y)-\gamma(y,x)\vert\,\mathrm{d}y\,\mathrm{d}x\\
        &+\frac{1}{2}\int_{\Omega}\int_{\widehat{\Gamma}}  (u(x)-u(y))^{2}\gamma(y,x)\,\mathrm{d}y\,\mathrm{d}x
    \end{align*}
    for $u\in\V$. In other words $\widetilde{\mathfrak{B}}$ is coercive on $\V$ and by the Lax--Milgram Theorem there exists a unique weak solution.
\end{proof}

%%=============================================%%

\section{Nonlocal Poincar\'e inequality}\label{sec:poin}

Similar to the local Neumann problem the Lax--Milgram Theorem provides us an existence and uniqueness result for the weak solution of problem \eqref{NP}. Note, that because in \Cref{thm:exriesz} the bilinear form is symmetric, we could directly apply Riesz representation theorem.

In the local Neumann problem the Poincar\'e inequality (see \cite[Chapter 5.8.1.]{evans2010partial}), namely there is a constant $C>0$ such that for all $u\in\mathrm{H}^{1}(\Omega)$ we have
\begin{align*}
    2\lambda(\Omega)\int_{\Omega}u^{2}(x)\,\mathrm{d}x-2\left(\int_{\Omega}u(y)\,\mathrm{d}y\right)^{2}=&\int_{\Omega}\int_{\Omega}(u(x)-u(y))^{2}\,\mathrm{d}y\,\mathrm{d}x\\
    \leqslant &C\int_{\Omega}\|\nabla u(x)\|^{2}\;\mathrm{d}x,
\end{align*}
is most commonly used in order to show that the corresponding bilinear form is coercive.
In the literature the inequalities most commonly called nonlocal Poincar\'e type inequality are either in the shape of (see for example \cite{hitch})
\begin{align*}
    \int_{\Omega}\int_{\Omega}(u(x)-u(y))^{2}\,\mathrm{d}y\,\mathrm{d}x\leqslant C \int_{\Omega}\int_{\Omega }(u(x)-u(y))^{2}\gamma(y,x)\,\mathrm{d}y\,\mathrm{d}x
\end{align*}
or (see for example \cite{gunz_1})
\begin{align*}
    \int_{\Omega}\int_{\Omega}(u(x)-u(y))^{2}\,\mathrm{d}y\,\mathrm{d}x\leqslant C \int_{\Omega\cup \Omega_{I}}\int_{\Omega\cup \Omega_{I}}(u(x)-u(y))^{2}\gamma(y,x)\,\mathrm{d}y\,\mathrm{d}x,
\end{align*}
where $\gamma\in\kernel$ and 
\begin{equation*}
    \Omega_{I}:=\{y\in\Rd\setminus\Omega \text{ such that there is a }  x\in\Omega \text{ with } \gamma(x,y)>0 \}.
\end{equation*}
We however refer
\begin{align*}
    \int_{\Omega}\int_{\Omega}(u(x)-u(y))^{2}\,\mathrm{d}y\,\mathrm{d}x\leqslant C \int_{\Omega}\int_{\Rd}(u(x)-u(y))^{2}\gamma(y,x)\,\mathrm{d}y\,\mathrm{d}x
\end{align*}
as nonlocal Poincar\'e  inequality. We easily see that
\begin{align*}
    \int_{\Omega}\int_{\Omega}(u(x)-u(y))^{2}\gamma(y,x)\,\mathrm{d}y\,\mathrm{d}x&\leqslant  \int_{\Omega}\int_{\Rd}(u(x)-u(y))^{2}\gamma(y,x)\,\mathrm{d}y\,\mathrm{d}x\\
    &\leqslant\int_{\Omega\cup\Omega_{I}}\int_{\Omega\cup\Omega_{I}}(u(x)-u(y))^{2}\gamma(y,x)\,\mathrm{d}y\,\mathrm{d}x
\end{align*}
holds. 
\begin{definition}
    Let $\Omega\subset\Rd$ be a nonempty, open and bounded set and $\gamma\in\kernel$. We say that the nonlocal Poincar\'e inequality holds (on $\V$), if there is a constant $C>0$ such that for all $u\in\V$ we have
    \begin{align*}
    	\int_{\Omega}\int_{\Omega}(u(x)-u(y))^{2}\,\mathrm{d}y\,\mathrm{d}x\leqslant C\int_{\Omega}\int_{\Rd}(u(x)-u(y))^{2}\gamma(y,x)\,\mathrm{d}y\,\mathrm{d}x.
    \end{align*}
    Every constant $C>0$ for which the nonlocal Poincar\'e inequality is satisfied is called Poincar\'e constant.
\end{definition}
Recalling that for every $\eta\in\kernel$ we set 
\begin{align*}
	\mathfrak{B}_{\eta}(u,v):=&\frac{1}{2}\int_{\Omega}\int_{\Omega}(u(x)-u(y))(v(x)-v(y))\,\eta(y,x)\,\mathrm{d}y\,\mathrm{d}x\\
	&+\int_{\Omega}\int_{\Gamma}(u(x)-u(y))(v(x)-v(y))\,\eta(y,x)\,\mathrm{d}y\,\mathrm{d}x
\end{align*}
for $u,v\in\mathrm{V}(\Omega;\eta)$, we obtain
\begin{lemma}\label{lem:poeq}
    Let $\Omega\subset\Rd$ be a nonempty, open and bounded set and $\gamma\in\kernel$. Then the following are equivalent:
     \begin{enumerate}[label=(\roman*)]
         \item The nonlocal Poincar\'e inequality holds.
         \item There is a constant $C>0$ such that for all $u\in\V$ we have\\ $\mathfrak{B}_{\chi_{\Omega\times \Omega}}(u,u)\leqslant C \mathfrak{B}_{\gamma}(u,u)$.
         \item There is a constant $C>0$ such that for all $u\in\V$ we have\\ $\|u-u_{\Omega}\|_{\Ls}\leqslant C \mathfrak{B}_{\gamma}(u,u)$, where $u_{\Omega}=\frac{1}{\lambda(\Omega)}\int_{\Omega}u(x)\,\mathrm{d}x$.
     \end{enumerate}
\end{lemma}
\begin{proof}
    Let $u\in\V$, then the binomial theorem and linearity of the integral yields
    \begin{align*}
        2\lambda(\Omega)\,\|u-u_{\Omega}\|_{\Ls}^{2}=&2\lambda(\Omega)\int_{\Omega}(u^{2}(x)-2 u_{\Omega} u(x)+(u_{\Omega})^{2})\,\mathrm{d}x\\
        =&2\lambda(\Omega)\int_{\Omega}u^{2}(x)\,\mathrm{d}x-4\left(\int_{\Omega}u(y)\,\mathrm{d}y\right)^{2}+2 \lambda(\Omega)^{2}(u_{\Omega})^{2}\\
        =&2\lambda(\Omega)\int_{\Omega}u^{2}(x)\,\mathrm{d}x-2\left(\int_{\Omega}u(y)\,\mathrm{d}y\right)^{2}\\
        =&\int_{\Omega}\int_{\Omega}(u^{2}(x)-2(x)u(y)+u^{2}(y))\,\mathrm{d}y\,\mathrm{d}x\\
        =&\int_{\Omega}\int_{\Omega}(u(x)-u(y))^{2}\,\mathrm{d}y\,\mathrm{d}x\\
        =&\mathfrak{B}_{\chi_{\Omega\times \Omega}}(u,u).
    \end{align*}
    And because of 
    \begin{align*}
        \frac{1}{2}\int_{\Omega}\int_{\Rd}(u(x)-u(y))^{2}\gamma(x,y)\,\mathrm{d}y\,\mathrm{d}x\leqslant \mathfrak{B}_{\gamma}(u,u)\leqslant \int_{\Omega}\int_{\Rd}(u(x)-u(y))^{2}\gamma(x,y)\,\mathrm{d}y\,\mathrm{d}x
    \end{align*}
    we get the equivalences.
\end{proof}
\begin{remark}
    Let $\Omega\subset\Rd$ be a nonempty, open and bounded set, then we have for all measurable functions $u\colon \Rd\to \mathbb{R}$ with $\mathfrak{B}_{\chi_{\Omega\times \Omega}}(u,u)<\infty$ that $u|_{\Omega}\in\mathrm{L}^{2}(\Omega)$ holds (see proof of Proposition 2.1 in \cite{mosco}). Furthermore this implies that the nonlocal Poincar\'e inequality holds on $\V$ if and only if there is a constant $C>0$ such that for every measurable function $v\colon\Rd\to\mathbb{R}$ with $v|_{\Omega}\in\mathrm{L}^{2}(\Omega)$ we have 
     \begin{align*}
    	\int_{\Omega}\int_{\Omega}(v(x)-v(y))^{2}\,\mathrm{d}y\,\mathrm{d}x\leqslant C\int_{\Omega}\int_{\Rd}(v(x)-v(y))^{2}\gamma(y,x)\,\mathrm{d}y\,\mathrm{d}x.
    \end{align*}
\end{remark}
Now we give some sufficient assumptions such that  the nonlocal Poincar\'e inequality holds:
\begin{theorem}\label{thm:poallg}
    Let $\Omega\subset\Rd$ be a nonempty, open and bounded set and $\gamma\in\kernel$ with
    \begin{align*}
        0<\int_{\Rd}\einf_{x\in\Omega}\gamma(y,x)\,\mathrm{d}y,
    \end{align*}
    then the nonlocal Poincar\'e inequality holds.
\end{theorem}
\begin{proof}
    Choose a constant $C>0$ such that there is a bounded measurable set 
    \begin{equation*}
        A\subset\{y\in\Rd\colon\einf_{x\in\Omega}\gamma(y,x)<C\}\text{ with } 0<c:=\int_{A}\einf_{x\in\Omega}\gamma(y,x)\,\mathrm{d}y,
    \end{equation*}
    then for every $u\in\V$ Jensen's inequality yields
    \begin{align*}
        &\int_{\Omega}\int_{\Omega}(u(x)-u(y))^{2}\,\mathrm{d}y\,\mathrm{d}x\\
        =&\frac{1}{c}\int_{\Omega}\int_{\Omega}\int_{A}(u(x)-u(t)+u(t)-u(y))^{2}\einf_{s\in\Omega}(\gamma(t,s))\,\mathrm{d}t\,\mathrm{d}y\,\mathrm{d}x\\
        \leqslant &\frac{2\lambda(\Omega)}{c}\int_{\Omega}\int_{\Rd}(u(x)-u(y))^{2}\gamma(y,x)\,\mathrm{d}y\,\mathrm{d}x.
    \end{align*}
\end{proof}
Due to Theorem 6.7 in \cite{hitch} we see, that the assumptions in \Cref{thm:poallg} are only sufficient conditions for the nonlocal Poincar\'e inequality. For this reason we now relax the assumptions on $\gamma$ by using this Lemma:
\begin{lemma}\label{lem:eps_abs}
    Let $\Omega\subset \Rd$ be a bounded domain and $\varepsilon>0$. Then there is a $C>0$ such that for all $u\in\Ls$ we have
    \begin{align*}
        \int_{\Omega}\int_{\Omega}(u(x)-u(y))^{2}\,\mathrm{d}y\,\mathrm{d}x\leqslant C\,\int_{\Omega}\int_{\Omega}(u(x)-u(y))^{2}\chi_{\{\|x-y\|<\varepsilon\}}\,\mathrm{d}y\,\mathrm{d}x.
    \end{align*}
\end{lemma}
\begin{proof}
    Due to the Heine–Borel Theorem we know that $\overline{\Omega}$ is compact. Meaning that there is a $N\in \mathbb{N}$ and a sequence $(\omega_{i})_{i\in\mathbb{N}}$ in $\Omega$ such that
    \begin{align*}
	    \overline{\Omega}\subset \bigcup\limits_{i\in\mathbb{N},i\leqslant N}\mathrm{B}_{\frac{\varepsilon}{4}}(\omega_{i})\quad\text{and}\quad \|\omega_{j+1}-\omega_{j}\|\leqslant \frac{\varepsilon}{2} \text{ for }j=1,\ldots,N-1. 
    \end{align*} 
    Let $x,y\in\Omega$, then, without loss of generality, we now assume, that $x\in \mathrm{B}_{\frac{\varepsilon}{4}}(\omega_{1})$. Let $j\in\{1,\ldots,N\}$ satisfy $y\in \mathrm{B}_{\frac{\varepsilon}{4}}(\omega_{j})$ and let $y_{\ell}\in \mathrm{B}_{\frac{\varepsilon}{4}}(0)$ for $\ell=1,\ldots,N$. Then
    \begin{align*}
    	&\left\|x-\omega_{1}+y_{1}\right\|\leqslant \frac{\varepsilon}{4}+\frac{\varepsilon}{4}\leqslant\varepsilon,\quad\left\|\omega_{j}-y+y_{j}\right\|\leqslant \frac{\varepsilon}{4}+\frac{\varepsilon}{4}\leqslant\varepsilon\\
	    \text{and}\quad&\left\|\omega_{i}+y_{i}-w_{i+1}-y_{i+1}\right\|\leqslant\left\|\omega_{i}-\omega_{i+1}\right\|+\left\|y_{i}\right\|+\left\|y_{i+1}\right\|\leqslant \varepsilon
    \end{align*}
    for $i=1,\ldots,\ell-1$. In other words for every $x,y\in\Omega$ there is a sequence $(z_{i})_{i\in\mathbb{N}}$ in $ \Omega$ such that 
    \begin{align*}
    	&\left\|x-z_{1}+y_{1}\right\|\leqslant \frac{\varepsilon}{4}+\frac{\varepsilon}{4}\leqslant\varepsilon,\quad\left\|z_{N}-y+y_{\ell}\right\|\leqslant \frac{\varepsilon}{4}+\frac{\varepsilon}{4}\leqslant\varepsilon\\
	   \text{and} \quad& \left\|z_{i}+y_{i}-z_{i+1}-y_{i+1}\right\|\leqslant \left\|z_{i}-z_{i+1}\right\|+\left\|y_{i}\right\|+\left\|y_{i+1}\right\|\leqslant \varepsilon
    \end{align*}
    for $i=1,\ldots,N-1$. Setting $D_{i}(x,y)= \left(\mathrm{B}_{\frac{\varepsilon}{4}}\left(z_{i}\right)\right)\cap\Omega$ for $i=1,\ldots,N$ we conclude that $(x_{1},\ldots,x_{N})\in\mathop{\mathsf{X}}\limits_{i=1}^{N} D_{i}(x,y)$ implies $\|x-x_{1}\|,\,\|x_{N}-y\|\,,\|x_{j}-x_{j+1}\|\leqslant\varepsilon$ for all $j=1,\ldots,N-1$ and therefore
    \begin{equation*}
        \begin{split}
           &\mathop{\mathsf{X}}\limits_{i=1}^{N} D_{i}(x,y)\\
           \subset&\{(v_{1},\ldots,v_{N})\in\Omega^{N}\colon \|x-v_{1}\|,\,\|v_{N}-y\|,\|v_{j}-v_{j+1}\|\leqslant\varepsilon \text{ for }j=1,\ldots,N-1\}.
        \end{split}
    \end{equation*}
    Because $\overline{\Omega}$ is compact and $\Omega$ is open there exists a constant $c>0$ for which $c\leqslant \inf\limits_{x\in\Omega}\lambda(\mathrm{B}_{\frac{\varepsilon}{4}}(x)\cap \Omega)$ holds. Hence we estimate
    \begin{align*}
    	&\int_{\Omega}\int_{\Omega}(u(x)-u(y))^{2}\,\mathrm{d}y\,\mathrm{d}x\\
    	\leqslant&\frac{1}{c^{N}} \int_{\Omega}\int_{D_{1}(x,y)}\int_{D_{2}(x,y)}\ldots\int_{D_{N}(x,y)}\int_{\Omega}(u(x)-u(y))^{2}\,\mathrm{d}y\,\mathrm{d}x_{N}\ldots\,\mathrm{d}x_{2}\,\mathrm{d}x_{1}\,\mathrm{d}x\\
    	\leqslant  &\frac{1}{c^{N}} \underbrace{ \int_{\Omega}\int_{\mathrm{B}_{\varepsilon}(x)}\int_{\mathrm{B}_{\varepsilon}(x_{1})}\ldots\int_{\mathrm{B}_{\varepsilon}(x_{N})}\int_{\Omega}(u(x)-u(y))^{2}\prod\limits_{i=1}^{N}\chi_{\Omega}(x_{i})\,\mathrm{d}y\,\mathrm{d}x_{N}\ldots\,\mathrm{d}x_{2}\,\mathrm{d}x_{1}\,\mathrm{d}x}_{=:J}
    \end{align*}
    For $j=1,2,\ldots,N-1$ set 
    \begin{align*}
        &J_{0}:= \int_{\Omega}\int_{\mathrm{B}_{\varepsilon}(x)}\ldots\int_{\mathrm{B}_{\varepsilon}(x_{N})}\int_{\Omega}(u(x)-u(x_{1}))^{2}\prod\limits_{i=1}^{N}\chi_{\Omega}(x_{i})\,\mathrm{d}y\,\mathrm{d}x_{N}\ldots\,\mathrm{d}x_{1}\,\mathrm{d}x,\\
        &J_{j}:=\int_{\Omega}\int_{\mathrm{B}_{\varepsilon}(x)}\ldots\int_{\mathrm{B}_{\varepsilon}(x_{N})}\int_{\Omega}(u(x_{j})-u(x_{j+1}))^{2}\prod\limits_{i=1}^{N}\chi_{\Omega}(x_{i})\,\mathrm{d}y\,\mathrm{d}x_{N}\ldots\,\mathrm{d}x_{1}\,\mathrm{d}x,\\
        \text{and}\,&J_{N}:=\int_{\Omega}\int_{\mathrm{B}_{\varepsilon}(x)}\ldots\int_{\mathrm{B}_{\varepsilon}(x_{N})}\int_{\Omega}(u(x_{N})-u(y))^{2}\prod\limits_{i=1}^{N}\chi_{\Omega}(x_{i})\,\mathrm{d}y\,\mathrm{d}x_{N}\ldots\,\mathrm{d}x_{1}\,\mathrm{d}x.
    \end{align*}
    Then for all  $j=0,1,2,\ldots,N$ we estimate
    \begin{equation*}
        J_{j}\leqslant (\lambda(\Omega))^{N}\displaystyle{\int_{\Omega}\int_{\|x-y\|\leqslant\varepsilon}(u(x)-u(y))^{2}\chi_{\Omega}(y)\,\mathrm{d}y\,\mathrm{d}x}.
    \end{equation*}
    And by iteratively applying Jensen's inequality we get 
    \begin{align*}
        &(u(x)-u(y))^{2}\\
        =&\left(u(x)-u(x_{1})+\left(\sum\limits_{j=1}^{N-1}(u(x_{j})-u(x_{j+1}))\right)+u(x_{N})-u(y)\right)^{2}\\
        \leqslant &(N+1)\left((u(x)-u(x_{1}))^{2}+\left(\sum\limits_{j=1}^{N-1}(u(x_{j})-u(x_{j+1}))^{2}\right)+(u(x_{N})-u(y))^{2}\right)
    \end{align*}
    and therefore  conclude
    \begin{align*}
    	J\leqslant &(N+1)\left(J_{0}+\left(\sum\limits_{j=1}^{N-1}J_{j}\right)+J_{N}\right)\\
    	\leqslant &(N+1)^{2}\left(\lambda(\Omega)\right)^{N} \int_{\Omega}\int_{\|x-y\|\leqslant\varepsilon}(u(x)-u(y))^{2}\chi_{\Omega}(y)\,\mathrm{d}y\,\mathrm{d}x.
    \end{align*}
\end{proof}
\begin{theorem}
    Let $\Omega\subset \Rd$ be a bounded domain, $\gamma\in\kernel$ and $\varepsilon<\infty$. Furthermore let $0<C_{1}\leqslant\gamma(x,y)$ hold for a.e.\;$x,y\in\Omega$ with $\|x-y\|\leqslant \varepsilon$. Then the nonlocal Poincar\'e inequality holds.
\end{theorem}
\begin{proof}
    Follows as a consequence of \Cref{lem:poeq} and \Cref{lem:eps_abs}.
\end{proof}
\begin{theorem}\label{thm:poeps}
    Let $\Omega\subset\Rd$ be a bounded, nonempty domain, $\varepsilon>0$ and $\gamma\in\kernel$ with
    \begin{align*}
        0<\einf_{x,y\in \Omega, \|x-y\|<\varepsilon}\int_{\Rd}\min\{\gamma(z,x),\gamma(z,y)\}\,\mathrm{d}z,
    \end{align*}
    then the nonlocal Poincar\'e inequality holds.
\end{theorem}
\begin{proof}
    Choose a constant $C>0$ such that there is a bounded measurable set 
    \begin{align*}
        & A\subset\{y\in\Rd\colon\min\{\gamma(z,x),\gamma(z,y)\}<C\text{ for a.e.\;}x,y\in\Omega\text{ with }\|x-y\|<\varepsilon\}\\
        &\text{with }0<c:=\einf_{x,y\in \Omega, \|x-y\|<\varepsilon}\int_{A}\min\{\gamma(z,x),\gamma(z,y)\}\,\mathrm{d}z.
    \end{align*}
    By \Cref{lem:eps_abs} there is a $C_{1}>0$ with
    \begin{equation*}
       \int_{\Omega}\int_{\Omega}(u(x)-u(y))^{2}\,\mathrm{d}y\,\mathrm{d}x\leqslant C_{1}\int_{\Omega}\int_{\Omega}(u(x)-u(y))^{2}\chi_{\{\|x-y\|<\varepsilon\}}\,\mathrm{d}y\,\mathrm{d}x\text{ for }u\in\V.
     \end{equation*}
    Hence for $u\in\V$ Jensen's inequality yields
     \begin{align*}
        &C_{1}\int_{\Omega}\int_{\Omega}(u(x)-u(y))^{2}\chi_{\{\|x-y\|<\varepsilon\}}\,\mathrm{d}y\,\mathrm{d}x\\
       \leqslant &\frac{C_{1}}{c}\int_{\Omega}\int_{\Omega}\int_{A}\min\{\gamma(z,x),\gamma(z,y)\}(u(x)-u(z)+u(z)-u(y))^{2}\,\mathrm{d}z\,\mathrm{d}y\,\mathrm{d}x\\
       \leqslant &\frac{2C_{1}\lambda(\Omega)}{c}\int_{\Omega}\int_{\Rd}(u(x)-u(z))^{2}\gamma(z,x)\,\mathrm{d}z\,\mathrm{d}x.
    \end{align*}
\end{proof}
\begin{corollary}
    Let $\Omega\subset\Rd$ be a bounded, nonempty domain and let $\gamma\in\kernel$ satisfy 
    \begin{equation*}
        0<\gamma_{0}<\gamma(y,x)\quad\text{ for all }\quad 0<r_{0}<\|x-y\|<r_{1}<\infty,
    \end{equation*}
    then the nonlocal Poincar\'e inequality holds.
\end{corollary}
\begin{proof}
    For $x\in \Omega$ we define $A_{x}=\{y\in \Rd\colon r_{0}<\|x-y\|<r_{1}\}$ and for $y,z\in\Rd$ we set $\widetilde{\gamma}(y,z)=\chi_{A_{z}}(y)\gamma_{0}$. Then we choose $\varepsilon>0$ such that 
    \begin{equation*}
        0<\einf_{x,y\in \Omega, \|x-y\|<\varepsilon}\lambda(A_{x}\cap A_{y})=\einf_{x,y\in \Omega, \|x-y\|<\varepsilon}\int_{\Rd}\min\{\widetilde{\gamma}(z,y),\widetilde{\gamma}(z,x)\}\,\mathrm{d}z,
    \end{equation*}
    holds and by \Cref{thm:poeps} we get 
    \begin{align*}
        \int_{\Omega}\int_{\Omega}(u(x)-u(y))^{2}\,\mathrm{d}y\,\mathrm{d}x&\leqslant C\int_{\Omega}\int_{A_{x}}(u(x)-u(y))^{2}\widetilde{\gamma}(y,x)\,\mathrm{d}y\,\mathrm{d}x\\ 
        &\leqslant C\int_{\Omega}\int_{\Rd}(u(x)-u(y))^{2}\gamma(y,x)\,\mathrm{d}y\,\mathrm{d}x\;\text{for }u\in \V.
    \end{align*}
\end{proof}
We define 
\begin{equation*}
    \Vd:=\{u\in\V \text{ such that }u=0 \text{ on }\Rd\setminus\Omega\},
\end{equation*}
then we have for all $u\in\Vd$ that
\begin{equation}\label{simp}
    \begin{split}
        &\int_{\Omega}\int_{\Rd}(u(x)-u(y))^{2}\gamma(y,x)\,\mathrm{d}y\,\mathrm{d}x\\
        =&\int_{\Omega}u^{2}(x)\left(\int_{\Gamma}\gamma(y,x)\,\mathrm{d}y\right)\,\mathrm{d}x+\int_{\Omega}\int_{\Omega}(u(x)-u(y))^{2}\gamma(y,x)\,\mathrm{d}y\,\mathrm{d}x\\
        <&\infty.
    \end{split}
\end{equation}
We say the nonlocal Friedrich's inequality holds on $\Vd$, if there is a $C>0$ such that 
\begin{equation*}
    \int_{\Omega}u^{2}(x)\,\mathrm{d}x\leqslant C \int_{\Omega}\int_{\Rd}(u(x)-u(y))^{2}\gamma(y,x)\,\mathrm{d}y\,\mathrm{d}x\quad\text{ holds for all }u\in \Vd.
\end{equation*}
By equation \eqref{simp} we see that the Friedrich's inequality holds on $\Vd$, if we have
\begin{align*}
    \einf_{x\in\Omega}\int_{\Gamma}\gamma(y,x)\,\mathrm{d}y>0.
\end{align*}
However this assumption can be relaxed.
\begin{theorem}\label{thm:fried}
    Let $\Omega\subset\Rd$ be a nonempty, open and bounded set and $\gamma\in\kernel$ such that there is a measurable $\widetilde{\Omega}\subset \Omega$ with
    \begin{align*}
        \einf_{x\in\widetilde{\Omega}}\int_{\Gamma}\gamma(y,x)\,\mathrm{d}y>0\quad\text{and}\quad \einf_{(y,x)\in\widetilde{\Omega}\times\Omega}\gamma(y,x)>0.
    \end{align*}
    Then the nonlocal Friedrich's inequality holds. Furthermore if the nonlocal Friedrich's inequality holds, then $\Vd$ is a Hilbert space with respect to 
    \begin{equation*}
        \langle u,v\rangle_{0}= \int_{\Omega}\int_{\Rd}(u(x)-u(y))(v(x)-v(y))\gamma(y,x)\,\mathrm{d}y\,\mathrm{d}x,\quad u,v\in\Vd.
    \end{equation*}
\end{theorem}
\begin{proof}
    For the first part let $u\in\Vd$, then we have
    \begin{align*}
        \einf_{x\in\widetilde{\Omega}}\int_{\Gamma}\gamma(y,x)\,\mathrm{d}y\int_{\widetilde{\Omega}}u^{2}(x)\,\mathrm{d}x\leqslant  &\int_{\widetilde{\Omega}}u^{2}(x)\int_{\Gamma}\gamma(y,x)\,\mathrm{d}y\,\mathrm{d}x\\
        \leqslant  &\int_{\Omega}\int_{\Gamma}(u(x)-u(y))^{2}\gamma(y,x)\,\mathrm{d}y\,\mathrm{d}x \\
        \leqslant  &\int_{\Omega}\int_{\Rd}(u(x)-u(y))^{2}\gamma(y,x)\,\mathrm{d}y\,\mathrm{d}x
    \end{align*}
    and
    \begin{align*}
        \lambda(\widetilde{\Omega})\int_{\Omega}u^{2}(x)\,\mathrm{d}x=&\int_{\Omega}\int_{\widetilde{\Omega}}(u(x)-u(y)+u(y))^{2}\,\mathrm{d}y\,\mathrm{d}x\\
        \leqslant &\int_{\Omega}\int_{\widetilde{\Omega}}2(u(x)-u(y))^{2}+2u^{2}(y)\,\mathrm{d}y\,\mathrm{d}x.
    \end{align*}
    Therefore the Friedrich's inequality is satisfied. And if the Friedrich's inequality is satisfied, then there is a $\alpha>0$ and $\beta<\infty$ such that for all $u\in\Vd$ we have
    \begin{align*}
        \alpha\|u\|_{\V}\leqslant\langle u,u\rangle_{0}\leqslant \beta\|u\|_{\V}.
    \end{align*}
    This means that $\langle \cdot,\cdot\rangle_{0}$ defines an inner product on $\Vd$.
\end{proof}

\begin{corollary}
    Let $\Omega\subset\Rd$ be nonempty, open and bounded and $s\in(0,1)$, then the assumption for the  Friedrich's inequality in \Cref{thm:fried} is satisfied for $\gamma_{s}(x,y):=\frac{1}{\|x-y\|^{d+2s}}$, $x,y\in\Rd$.
\end{corollary}
\begin{proof}
    See Lemma A.1. in \cite{SAVIN20141}.
\end{proof}

\section{Trace space}\label{sec:tr}

In this section we study the nonlocal trace space (for the local trace space see Leoni \cite[Chapter 18.]{leoni}). In other words we study the ``restriction'' of the elements of $\V$ on $\Gamma$. Like the  elements of the Sobolev space, the elements of $\V$ are in fact equivalence classes, so pointwise evaluations of these elements are in general impractical. However the elements of $\V$ are defined a.e.\;on the nonlocal boundary $\Gamma$.  

Most recently, nonlocal trace spaces were introduced, for the fractional Laplacian kernel by Bersetche and Borthagaray \cite{tracefrac} as well as Dyda and Kassmann \cite{kasstrace}. Furthermore Tian and Du \cite{dutrace} studied a nonlocal Trace space, for regional kernels, by using density arguments. As in \cite{tracefrac}, \cite{von} and \cite{kasstrace}, we will prove the existence of a weighted Lebesgue space $\mathrm{L}^{2}(\Gamma;w)$, for which
\begin{equation*}
    \Tr\colon\V\to\mathrm{L}^{2}(\Gamma;w),\quad v\mapsto v\vert_{\Gamma}
\end{equation*}
is a continuous linear operator, i.e., there is a constant $C>0$ with
\begin{equation*}
    \|\Tr(u)\|_{\mathrm{L}^{2}(\Gamma;w)}\leqslant C \|u\|_{\V}\quad\text{ for all }u\in\V.
\end{equation*}
However, while the connection between the weighted Lebesgue space which we study and the one introduced in \cite{von} is clear, the connection between the nonlocal Trace space of \cite{tracefrac} and \cite{kasstrace} and the one we study remains an open question.

Also, we will find that the  measurable weight function $w\colon \Gamma\to(0,\infty)$ only depends on $\gamma$ and $\Omega$. Because we have $\Tr(u)=0$ for all $u\in\Vd$, we see that $\Tr$ is injective if and only if $\Vd=\{0\}$. As in \cite{tracesur} we present a characterization of the trace space for some example kernels. 

Note that by using Fubini's Theorem we see
\begin{align*}
    \int_{\Omega}\int_{\Gamma}\gamma(y,x)\,\mathrm{d}y\,\mathrm{d}x=\int_{\Gamma}\int_{\Omega}\gamma(y,x)\,\mathrm{d}y\,\mathrm{d}x.
\end{align*}
Therefore $\Gamma$ is a null set, if and only if the set $\{x\in\Omega \colon\int_{\Gamma}\gamma(y,x)\,\mathrm{d}y>0\}$ is a null set. Furthermore, if
\begin{equation*}
    \Omega=\{x\in\Omega\colon\int_{\Gamma}\gamma(y,x)\,\mathrm{d}y=\infty\}
\end{equation*}
holds, then we get $\Vd=\{0\}$ by \cref{simp}. 
\begin{theorem}\label{thm:Tr1}
    Let $\Omega\subset \Rd$ be a bounded, nonempty, open and $\gamma\in\kernel$. For a.e.\;$x\in \Omega$ we assume 
    \begin{align*}
        \int_{\Gamma}\gamma(z,x)\,\mathrm{d}z<\infty. 
    \end{align*}
    Furthermore let $c\in[0,\infty)$ satisfy $\einf\limits_{x\in\Omega}\int_{\Gamma}\gamma(y,x)\,\mathrm{d}y+c>0$ and define the function $w\colon \Gamma \to (0,\infty]$ by
    \begin{equation*}
        w(y)=\int_{\Omega}\dfrac{\gamma(y,x)}{\int_{\Gamma}\gamma(z,x)\,\mathrm{d}z+c}\,\mathrm{d}x.
    \end{equation*}
    Then $\{y\in\Gamma\colon\,w(y)=\infty\}$ is a null set and
    \begin{equation*}
       \Tr\colon\V\to\mathrm{L}^{2}(\Gamma;w),\, T(v)=v\vert_{\Gamma}
    \end{equation*}
    \text{is a continuous linear operator.}
\end{theorem}
\begin{proof}
    By definition we have $w(y)>0$ and because $\int_{\Gamma}w(y)\,\mathrm{d}y\leqslant\lambda(\Omega)$ holds, $w$ is finite almost everywhere. For $u\in\V$, the Jensen's inequality and Fubini's Theorem yield
    \begin{align*}
        \|\Tr(u)\|_{\mathrm{L}^{2}(\Gamma;w)}^{2}=&\int_{\Gamma}u^{2}(y)w(y)\,\mathrm{d}y\\
        =&\int_{\Gamma}\int_{\Omega}(u(x)-u(y)-u(x))^{2}\dfrac{\gamma(y,x)}{\int_{\Gamma}\gamma(z,x)\,\mathrm{d}z+c}\,\mathrm{d}x\,\mathrm{d}y\\
        \leqslant &2 \int_{\Omega}u^{2}(x)\,\mathrm{d}x+2\int_{\Omega}\int_{\Rd}(u(x)-u(y))^{2}\dfrac{\gamma(y,x)}{\int_{\Gamma}\gamma(z,x)\,\mathrm{d}z+c}\,\mathrm{d}y\,\mathrm{d}x\\
        \leqslant &2\max\left\{\dfrac{1}{\einf\limits_{x\in\Omega}\int_{\Gamma}\gamma(y,x)\,\mathrm{d}y+c},1\right\}\|u\|_{\V}^{2}.
    \end{align*}
\end{proof}

Now we want to characterize the Trace space. By taking a closer look at the proof of \cref{thm:Tr1} we see:
\begin{theorem}\label{thm:Tr2}
    Let $\Omega\subset \Rd$ be a bounded, nonempty, open and $\gamma\in\kernel$ with $\esup_{x\in\Omega}\int_{\Gamma}\gamma(y,x)\mathrm{d}y<\infty$. Define the function $w\colon \Gamma \to (0,\infty]$ by
    \begin{equation*}
        w(y)=\int_{\Omega}\gamma(y,x)\,\mathrm{d}x,
    \end{equation*}
    then $\{y\in\Gamma\colon\,w(y)=\infty\}$ is a null set,
    \begin{equation*}
        \Tr\colon\V\to\mathrm{L}^{2}(\Gamma;w), \; T(v)=v\vert_{\Gamma} \;\text{is a bounded, linear and surjective operator},
    \end{equation*}
    and
    \begin{equation*}
        \operatorname{Ext}\colon \mathrm{L}^{2}(\Gamma;w)\to\V,\; D(c)=c\chi_{\Gamma}\;\text{is a bounded, injective and linear operator}.
    \end{equation*}
    Furthermore for all $c\in\mathrm{L}^{2}(\Gamma;w)$ we have $\Tr(\operatorname{Ext}(c))=c$.
\end{theorem}
\begin{proof}
    Because $w\in\mathrm{L}^{1}(\Gamma)$, we see that $\{y\in\Gamma\colon\,w(y)=\infty\}$ is a null set. For $u\in\V$, Fubini's Theorem and the Jensen's inequality yield
    \begin{align*}
        &\|\Tr(u)\|_{\mathrm{L}^{2}(\Gamma;w)}^{2}\\
        =&\int_{\Gamma}u^{2}(y)w(y)\,\mathrm{d}y\\
        =&\int_{\Omega}\int_{\Gamma}(u(x)-u(y)-u(x))^{2}\gamma(y,x)\,\mathrm{d}y\,\mathrm{d}x\\
        \leqslant &2\esup_{x\in\Omega}\int_{\Gamma}\gamma(y,x)\,\mathrm{d}y \int_{\Omega}u^{2}(x)\,\mathrm{d}x+2\int_{\Omega}\int_{\Rd}(u(x)-u(y))^{2}\gamma(y,x)\,\mathrm{d}y\,\mathrm{d}x\\
        \leqslant &2\max\{\esup_{x\in\Omega}\int_{\Gamma}\gamma(y,x)\,\mathrm{d}y,1\}\|u\|_{\V}^{2}.
    \end{align*}
    Let $c\in\mathrm{L}^{2}(\Gamma;w)$, then we have 
    \begin{align*}
        \|\operatorname{Ext}(c)\|_{\V}^{2}= \int_{\Omega}\int_{\Rd}(c(y)\chi_{\Gamma}(y))^{2}\gamma(y,x)\,\mathrm{d}y\,\mathrm{d}x=\|c\|_{\mathrm{L}^{2}(\Gamma;w)}.
    \end{align*}
\end{proof}
\begin{corollary}
    Let the assumptions of either \Cref{thm:Tr1} or \Cref{thm:Tr2} be satisfied and set $w\colon\Gamma\to(0,\infty]$ accordingly. Then a measurable function $g\colon\Gamma\to\mathbb{R}$ satisfies the continuous functional condition if $\frac{g}{\sqrt{w}}\in\mathrm{L}^{2}(\Gamma)$ holds, in other words if we have 
    \begin{equation*}
        \left\|\frac{g}{\sqrt{w}}\right\|_{\mathrm{L}^{2}(\Gamma)}^{2}=\int_{\Gamma}\frac{g^{2}(y)}{w(y)}\mathrm{d}y<\infty.
    \end{equation*}
    Moreover, if we assume
	\begin{align*}
	    \einf_{y\in\Gamma}w(y)>0,
    \end{align*}
    then every $g\in\mathrm{L}^{2}(\Gamma)$ satisfies the continuous functional condition. In particular, if we have
    \begin{equation*}
        0<\einf_{y\in\Gamma}w(y)\leqslant\esup_{y\in\Gamma}w(y)<\infty,
    \end{equation*}
    then we obtain $u\vert_{\Omega\cup \Gamma}\in \mathrm{L}^{2}(\Omega\cup \Gamma)$ for all $u\in \V$.
\end{corollary}
\begin{proof}
    By the Hölder inequality we have, for all $ u\in\V$,
    \begin{equation*}
        \int_{\Gamma}g(y)u(y)\mathrm{d}y\leqslant\left\|\frac{g}{\sqrt{w}}\right\|_{\mathrm{L}^{2}(\Gamma)}\|\Tr(u)\|_{\mathrm{L}^{2}(\Gamma;w)}\leqslant\left\|\frac{g}{\sqrt{w}}\right\|_{\mathrm{L}^{2}(\Gamma)}\|u\|_{\V}.
    \end{equation*}
    The rest is a direct consequence of the bounds of $w$.
\end{proof}
Recalling the characterization of the local trace space (see \cite[Theorem 18.40]{leoni}) we obtain:
\begin{theorem}\label{thm:trsur}
    Let $\Omega\subset \Rd$ be a bounded, nonempty, open subset and $\gamma\in\kernel$ such that for a.e.\;$x\in \Omega$ we have 
    \begin{equation*}
        \int_{\Gamma}\gamma(z,x)\,\mathrm{d}z<\infty. 
    \end{equation*}
    Furthermore let $c\in[0,\infty)$ satisfy $\einf\limits_{x\in\Omega}\int_{\Gamma}\gamma(y,x)\,\mathrm{d}y+c>0$ and define the function $w\colon \Gamma \to (0,\infty]$ by
    \begin{equation*}
        w(y)=\int_{\Omega}\dfrac{\gamma(y,x)}{\int_{\Gamma}\gamma(z,x)\,\mathrm{d}z+c}\,\mathrm{d}x
    \end{equation*}
    and set $\mathrm{W}(\Gamma;\gamma):=\{u:\Gamma\to \mathbb{R}\quad\text{measurable with }\|u\|_{\mathrm{W}(\Gamma;\gamma)}<\infty\}$, where
    \begin{align*}
        \|u\|_{\mathrm{W}(\Gamma;\gamma)}^{2}:=\int_{\Gamma}u^{2}(y)w(y)\,\mathrm{d}y+\int_{\Gamma}\int_{\Gamma}(u(y)-u(z))^{2}\int_{\Omega}\dfrac{\gamma(y,x)\gamma(z,x)}{\int_{\Gamma}\gamma(s,x)\,\mathrm{d}s+c}\,\mathrm{d}x\,\mathrm{d}y\,\mathrm{d}z.
    \end{align*}
    Then:
    \begin{enumerate}[label=(\roman*)]
        \item $\Tr\colon\V\to\mathrm{W}(\Gamma;\gamma),\, \Tr(v)=v\vert_{\Gamma}$ is a linear operator such that, there is a $ C>0$ with 
        \begin{align*}
            \|\Tr (u)\|_{\mathrm{W}(\Gamma;\gamma)}\leqslant C \|u\|_{\V}\text{ for }u\in\V.
        \end{align*}
        \item $\operatorname{E}\colon\mathrm{W}(\Gamma;\gamma)\to\Ls,$
        \begin{equation*}
            \operatorname{E}(v)=\displaystyle{\int_{\Gamma}v(y)\frac{\gamma(y,\cdot)}{\int_{\Gamma}\gamma(s,\cdot)\,\mathrm{d}s+c}\,\mathrm{d}y}
        \end{equation*}
        is a linear operator such that, there is a $ C>0$ with 
        \begin{equation*}
            \|\operatorname{E}(v)\|_{\Ls}+\int_{\Omega}\int_{\Gamma}(E(v)(x)-v(y))^{2}\gamma(y,x)\,\mathrm{d}y\,\mathrm{d}x\leqslant C \|v\|_{\mathrm{W}(\Gamma;\gamma)}
        \end{equation*}
        for $v\in\mathrm{W}(\Gamma;\gamma)$.
        \item If there is a $ C>0$ with
        \begin{equation*}
           \int_{\Omega}\int_{\Omega}(E(v)(x)-E(v)(y))^{2}\gamma(y,x)\,\mathrm{d}y\,\mathrm{d}x\leqslant C \|v\|_{\mathrm{W}(\Gamma;\gamma)}\quad\text{ for }v\in\mathrm{W}(\Gamma;\gamma),
        \end{equation*}
        then $\Tr$ is surjective.
    \end{enumerate}
\end{theorem}
\begin{proof}
    Because $\int_{\Gamma}w(y)\,\mathrm{d}y\leqslant\lambda(\Omega)$ holds, $w$ is finite almost everywhere. Now let $u\in\V$, then we have already shown in \Cref{thm:Tr1} that
    \begin{align*}
        \|\Tr(u)\|_{\mathrm{L}^{2}(\Gamma;w)}^{2}=\int_{\Gamma}u^{2}(y)w(y)\,\mathrm{d}y\leqslant 2\max\left\{\dfrac{1}{\einf\limits_{x\in\Omega}\int_{\Gamma}\gamma(y,x)\,\mathrm{d}y+c},1\right\}\|u\|_{\V}^{2}.
    \end{align*}
    By using Jensen's inequality and Fubini's Theorem we get
    \begin{align*}
        &\int_{\Gamma}\int_{\Gamma}(u(y)-u(z))^{2}\int_{\Omega}\dfrac{\gamma(y,x)\gamma(z,x)}{\int_{\Gamma}\gamma(s,x)\,\mathrm{d}s+c}\,\mathrm{d}x\,\mathrm{d}y\,\mathrm{d}z\\
        =&\int_{\Gamma}\int_{\Gamma}\int_{\Omega}(u(y)-u(x)+u(x)-u(z))^{2}\dfrac{\gamma(y,x)\gamma(z,x)}{\int_{\Gamma}\gamma(s,x)\,\mathrm{d}s+c}\,\mathrm{d}x\,\mathrm{d}y\,\mathrm{d}z\\
        \leqslant&2\int_{\Gamma}\int_{\Omega}(u(y)-u(x))^{2}\dfrac{\gamma(y,x)\int_{\Gamma}\gamma(z,x)\,\mathrm{d}z}{\int_{\Gamma}\gamma(z,x)\,\mathrm{d}z+c}\,\mathrm{d}x\,\mathrm{d}y\\
        &+2\int_{\Gamma}\int_{\Omega}(u(y)-u(x))^{2}\dfrac{\gamma(z,x)\int_{\Gamma}\gamma(y,x)\,\mathrm{d}y}{\int_{\Gamma}\gamma(z,x)\,\mathrm{d}z+c}\,\mathrm{d}x\,\mathrm{d}z\\
        \leqslant &4\|u\|_{\V}.
    \end{align*}
    Now let $v\in\mathrm{W}(\Omega;\gamma)$, then the Hölder inequality yields
    \begin{align*}
        \|\operatorname{E}(v)\|_{\Ls}=&\int_{\Omega}\left(\int_{\Gamma}v(y)\frac{\gamma(y,x)}{\int_{\Gamma}\gamma(z,x)\,\mathrm{d}z+c}\,\mathrm{d}y\right)^{2}\,\mathrm{d}x\\
        \leqslant &\int_{\Gamma}v^{2}(y)\int_{\Omega}\frac{\gamma(y,x)}{\int_{\Gamma}\gamma(z,x)\,\mathrm{d}z+c}\,\mathrm{d}x\,\mathrm{d}y.
    \end{align*}
    and 
    \begin{align*}
        &\int_{\Omega}\int_{\Gamma}\left(E(v)(x)-v(y)\right)^{2}\gamma(y,x)\,\mathrm{d}y\,\mathrm{d}x\\
        =&\int_{\Omega}\int_{\Gamma}\left(\int_{\Gamma}v(s)\frac{\gamma(s,x)}{\int_{\Gamma}\gamma(z,x)\,\mathrm{d}z+c}\,\mathrm{d}s-v(y)\frac{\int_{\Gamma}\gamma(z,x)\,\mathrm{d}z+c}{\int_{\Gamma}\gamma(z,x)\,\mathrm{d}z+c}\right)^{2}\gamma(y,x)\,\mathrm{d}y\,\mathrm{d}x\\
        \leqslant&\int_{\Omega}\int_{\Gamma}2\left(\int_{\Gamma}v(s)-v(y)\frac{\gamma(s,x)}{\int_{\Gamma}\gamma(z,x)\,\mathrm{d}z+c}\,\mathrm{d}s\right)^{2}\gamma(y,x)\,\mathrm{d}y\,\mathrm{d}x\\
        &+\int_{\Omega}\int_{\Gamma}2v^{2}(y)\frac{c^{2}\gamma(y,x)}{(\int_{\Gamma}\gamma(z,x)\,\mathrm{d}z+c)^{2}}\,\mathrm{d}y\,\mathrm{d}x\\
        \leqslant &2\int_{\Gamma}\int_{\Gamma}\left(v(s)-v(y)\right)^{2}\int_{\Omega}\frac{\gamma(s,x)\gamma(y,x)}{\int_{\Gamma}\gamma(z,x)\,\mathrm{d}z+c}\,\mathrm{d}x\,\mathrm{d}s\,\mathrm{d}y\\
        &+2\int_{\Gamma}v^{2}(y)\int_{\Omega}\frac{\gamma(y,x)}{(c^{-1}\int_{\Gamma}\gamma(z,x)\,\mathrm{d}z+1)^{2}}\,\mathrm{d}x\,\mathrm{d}y\\
        \leqslant &2\int_{\Gamma}\int_{\Gamma}\left(v(s)-v(y)\right)^{2}\int_{\Omega}\frac{\gamma(s,x)\gamma(y,x)}{\int_{\Gamma}\gamma(z,x)\,\mathrm{d}z+c}\,\mathrm{d}x\,\mathrm{d}s\,\mathrm{d}y\\
        &+2\int_{\Gamma}v^{2}(y)\int_{\Omega}\frac{c\,\gamma(y,x)}{\int_{\Gamma}\gamma(z,x)\,\mathrm{d}z+c}\,\mathrm{d}x\,\mathrm{d}y.
    \end{align*}
\end{proof}
\begin{corollary}
    Let $\Omega\subset \Rd$ be a bounded, nonempty, open subset and $\gamma\in\kernel$ such that for a.e.\;$x\in \Omega$ we have
    \begin{equation*}
        \int_{\Gamma}\gamma(z,x)\,\mathrm{d}z<\infty
    \end{equation*}
    and such that, there is a $c\in[0,\infty)$ with $\einf_{x\in\Omega} \int_{\Gamma}\gamma(z,x)\,\mathrm{d}z+c>0$,
    \begin{align*}
        \esup_{x\in\Omega}\int_{\Omega}\int_{\Gamma}\frac{(k(s,x)-k(s,y))^{2}}{k(s,x)+k(s,y)}\chi_{\{k(s,x)+k(s,y)>0\}}\,\mathrm{d}s\,\gamma(y,x)\,\mathrm{d}y<\infty\\
        \text{and}\quad \esup_{y\in\Omega}\int_{\Omega}\int_{\Gamma}\frac{(k(s,x)-k(s,y))^{2}}{k(s,x)+k(s,y)}\chi_{\{k(s,x)+k(s,y)>0\}}\,\mathrm{d}s\,\gamma(y,x)\,\mathrm{d}y<\infty,
    \end{align*}
    where $k(s,x)=\dfrac{\gamma(s,x)}{\int_{\Gamma}\gamma(z,x)\,\mathrm{d}z+c}$ for $(s,x)\in\Gamma\times\Omega$. Define the function $w\colon \Gamma \to (0,\infty]$ by
    \begin{equation*}
        w(y):=\int_{\Omega}k(y,x)\,\mathrm{d}x=\int_{\Omega}\dfrac{\gamma(y,x)}{\int_{\Gamma}\gamma(z,x)\,\mathrm{d}z+c}\,\mathrm{d}x
    \end{equation*}
    and set $\mathrm{W}(\Gamma;\gamma):=\{u:\Gamma\to \mathbb{R}\quad\text{measurable with }\|u\|_{\mathrm{W}(\Gamma;\gamma)}<\infty\}$, where
    \begin{align*}
        \|u\|_{\mathrm{W}(\Gamma;\gamma)}^{2}:=\int_{\Gamma}u^{2}(y)w(y)\,\mathrm{d}y+\int_{\Gamma}\int_{\Gamma}(u(y)-u(z))^{2}\int_{\Omega}\dfrac{\gamma(y,x)\gamma(z,x)}{\int_{\Gamma}\gamma(s,x)\,\mathrm{d}s+c}\,\mathrm{d}x\,\mathrm{d}y\,\mathrm{d}z.
    \end{align*}
    Then $\Tr\colon\V\to\mathrm{W}(\Gamma;\gamma),\, \Tr(v)=v\vert_{\Gamma}$ is a bounded, linear and surjective operator.
\end{corollary}
\begin{proof}
    Due to \Cref{thm:trsur}, it remains to show that there is a $ C>0$ with
    \begin{equation*}
       \int_{\Omega}\int_{\Omega}(E(v)(x)-E(v)(y))^{2}\gamma(y,x)\,\mathrm{d}y\,\mathrm{d}x\leqslant C \|v\|_{\mathrm{W}(\Gamma;\gamma)}.
    \end{equation*}
    First we mention that $k$ is nonegative. Hence for a.e.\;$x,y\in\Omega$ and a.e.\;$s\in\Gamma$ we have $k(s,x)+k(s,y)=0$ if and only if $k(s,x)=k(s,y)=0$ holds. Choose $C>0$ with
    \begin{align*}
        \esup_{x\in\Omega}\int_{\Omega}\int_{\Gamma}\frac{(k(s,x)-k(s,y))^{2}}{k(s,x)+k(s,y)}\chi_{\{k(s,x)+k(s,y)>0\}}\,\mathrm{d}s\,\gamma(y,x)\,\mathrm{d}y<\frac{C}{2}\\
        \text{and}\quad \esup_{y\in\Omega}\int_{\Omega}\int_{\Gamma}\frac{(k(s,x)-k(s,y))^{2}}{k(s,x)+k(s,y)}\chi_{\{k(s,x)+k(s,y)>0\}}\,\mathrm{d}s\,\gamma(y,x)\,\mathrm{d}y<\frac{C}{2}.
    \end{align*}
    Let $v\in\mathrm{W}(\Gamma;\gamma)$, then  we get by Hölder's inequality
    \begin{align*}
        &(E(v)(x)-E(v)(y))^{2}\\
        =&\left(\int_{\Gamma}v(s)\left(\frac{\gamma(s,x)}{\int_{\Gamma}\gamma(z,x)\,\mathrm{d}z+c}-\frac{\gamma(s,y)}{\int_{\Gamma}\gamma(z,y)\,\mathrm{d}z+c}\right)\,\mathrm{d}s\right)^{2}\\
        =&\left(\int_{\Gamma}v(s)\sqrt{\frac{k(s,x)+k(s,y)}{k(s,x)+k(s,y)}}(k(s,x)-k(s,y))\chi_{\{k(s,x)+k(s,y)>0\}}\,\mathrm{d}s\right)^{2}\\
        \leqslant&\int_{\Gamma}v^{2}(t)(k(t,x)+k(t,y))\,\mathrm{d}t\int_{\Gamma}\frac{(k(s,x)-k(s,y))^{2}}{k(s,x)+k(s,y)}\chi_{\{k(s,x)+k(s,y)>0\}}\,\mathrm{d}s
    \end{align*}
    for a.e.\;$x,y\in\Omega$. Therefore we conclude 
    \begin{align*}
        &\int_{\Omega}\int_{\Omega}(E(v)(x)-E(v)(y))^{2}\gamma(y,x)\,\mathrm{d}y\,\mathrm{d}x\\
        \leqslant &\int_{\Gamma}v^{2}(t)\int_{\Omega}k(t,x)\int_{\Omega}\int_{\Gamma}\frac{(k(s,x)-k(s,y))^{2}}{k(s,x)+k(s,y)}\chi_{\{k(s,x)+k(s,y)>0\}}\,\mathrm{d}s\gamma(y,x)\,\mathrm{d}y\,\mathrm{d}x\,\mathrm{d}t\\
        &+\int_{\Gamma}v^{2}(t)\int_{\Omega}k(t,y)\int_{\Omega}\int_{\Gamma}\frac{(k(s,x)-k(s,y))^{2}}{k(s,x)+k(s,y)}\chi_{\{k(s,x)+k(s,y)>0\}}\,\mathrm{d}s\gamma(y,x)\,\mathrm{d}x\,\mathrm{d}y\,\mathrm{d}t\\
        \leqslant &C\int_{\Gamma}v^{2}(s)\int_{\Omega}\frac{\gamma(s,x)}{\int_{\Gamma}\gamma(z,x)\,\mathrm{d}z+c}\,\mathrm{d}x\,\mathrm{d}s\\
        \leqslant  &C\|v\|_{\mathrm{W}(\Gamma;\gamma)}.
    \end{align*}
\end{proof}
\begin{corollary}
    Let $\Omega\subset \Rd$ be a bounded, convex, nonempty, open subset and $\gamma\in\kernel$ such that for a.e.\;$x\in \Omega$ we have
    \begin{equation*}
        \int_{\Gamma}\gamma(z,x)\,\mathrm{d}x<\infty
    \end{equation*}
    and such that $\frac{\gamma(y,\cdot)}{\int_{\Gamma}\gamma(z,\cdot)\,\mathrm{d}z}$ is differentiable a.e.\;in $\Omega$ for a.e.\;$y\in\Gamma$. Furthermore let there be a measurable function $\varphi\colon\Rd\to [0,\infty]$ and a constant $k>0$ with $\gamma(y,x)\leqslant k \varphi(y-x)$ for a.e.\;$x,y\in\Omega$ and
    \begin{equation*}
        \int_{\Rd}\min\{1,\|z\|^{2}\}\varphi(z)\,\mathrm{d}z<\infty.
    \end{equation*}
    Let $c\in[0,\infty)$ satisfy $\einf\limits_{x\in\Omega}\int_{\Gamma}\gamma(y,x)\,\mathrm{d}y+c>0$ and define the function $w\colon \Gamma \to (0,\infty]$ by
    \begin{equation*}
        w(y)=\int_{\Omega}\dfrac{\gamma(y,x)}{\int_{\Gamma}\gamma(z,x)\,\mathrm{d}z+c}\,\mathrm{d}x
    \end{equation*}
    and set $\mathrm{W}(\Gamma;\gamma):=\{u:\Gamma\to \mathbb{R}\quad\text{measurable with }\|u\|_{\mathrm{W}(\Gamma;\gamma)}<\infty\}$, where
    \begin{align*}
        \|u\|_{\mathrm{W}(\Gamma;\gamma)}^{2}:=\int_{\Gamma}u^{2}(y)w(y)\,\mathrm{d}y+\int_{\Gamma}\int_{\Gamma}(u(y)-u(z))^{2}\int_{\Omega}\dfrac{\gamma(y,x)\gamma(z,x)}{\int_{\Gamma}\gamma(s,x)\,\mathrm{d}s+c}\,\mathrm{d}x\,\mathrm{d}y\,\mathrm{d}z.
    \end{align*}
    If there is a constant $C>0$ with
    \begin{equation*}
       \|\frac{\partial}{\partial x}\left(\frac{\gamma(y,\cdot)}{\int_{\Gamma}\gamma(z,\cdot)\,\mathrm{d}z+c}\right)\|\leqslant C\left(\frac{\gamma(y,\cdot)}{\int_{\Gamma}\gamma(z,\cdot)\,\mathrm{d}z+c}\right) \text{ a.e.\;in }\Omega\text{ for a.e.\:}y\in\Gamma,  
    \end{equation*}
    then $\Tr\colon\V\to\mathrm{W}(\Gamma;\gamma),\, \Tr(v)=v\vert_{\Gamma}$ is a bounded, linear and surjective operator.
\end{corollary}
\begin{proof}
    For a.e.\;$s\in\Gamma$ and a.e.\;$x,y\in\Omega$ we get
    \begin{align*}
        &\left\vert\frac{\gamma(s,x)}{\int_{\Gamma}\gamma(z,x)\,\mathrm{d}z+c}-\frac{\gamma(s,y)}{\int_{\Gamma}\gamma(z,y)\,\mathrm{d}z+c}\right\vert\\
        =&\left\vert\int_{(0,1)}\langle\frac{\partial}{\partial x}\frac{\gamma(s,x+t(y-x))}{\int_{\Gamma}\gamma(z,x+t(y-x))\,\mathrm{d}z+c},y-x\rangle\,\mathrm{d}t\right\vert\\
        \leqslant&C \int_{(0,1)}\frac{\gamma(s,x+t(y-x))}{\int_{\Gamma}\gamma(z,x+t(y-x))\,\mathrm{d}z+c} \|y-x\|\,\mathrm{d}t.
    \end{align*}
    Let $v\in\mathrm{W}(\Gamma;\gamma)$. By the Hölder inequality we estimate
    \begin{align*}
        &\int_{\Omega}\int_{\Omega}(E(v)(x)-E(v)(y))^{2}\gamma(y,x)\,\mathrm{d}y\,\mathrm{d}x\\
        \leqslant &C\int_{\Gamma}v^{2}(s)\int_{\Omega}\int_{\Omega}\int_{(0,1)}\frac{\gamma(s,x+t(y-x))}{\int_{\Gamma}\gamma(z,x+t(y-x))\,\mathrm{d}z+c} \|y-x\|^{2}\gamma(y,x)\,\mathrm{d}t\,\mathrm{d}y\,\mathrm{d}x\,\mathrm{d}s.
    \end{align*}
    Because $\Omega$ is bounded, there is a $R>1$ with $\|x-y\|\leqslant R$ for all $x,y\in\Omega$ and substitution yields for a.e $s\in\Gamma$
    \begin{align*}
        &\int_{\Omega} \int_{\Omega}\int_{(0,1)}\frac{\gamma(s,x+t(y-x))}{\int_{\Gamma}\gamma(z,x+t(y-x))\,\mathrm{d}z+c} \|y-x\|^{2}\gamma(y,x)\,\mathrm{d}t\,\mathrm{d}y\,\mathrm{d}x\\
        \leqslant& k\int_{\Omega}\frac{\gamma(s,a)}{\int_{\Gamma}\gamma(z,a)\,\mathrm{d}z+c}  \int_{\Omega}\int_{(0,1)}\min\{\|\frac{a-x}{t}\|^{2},R\}\varphi(\frac{a-x}{t})\,\mathrm{d}t\,\mathrm{d}x\,\mathrm{d}a\\
        \leqslant&k R  \int_{\Rd}\min\{1,\|x\|^{2}\}\varphi(x)\,\mathrm{d}z\int_{\Omega}\frac{\gamma(s,a)}{\int_{\Gamma}\gamma(z,a)\,\mathrm{d}z+c}\,\mathrm{d}a .
    \end{align*}
    All in all we conclude, there is a $\alpha>0$ with
    \begin{equation*}
        \int_{\Omega}\int_{\Omega}(E(v)(x)-E(v)(y))^{2}\gamma(y,x)\,\mathrm{d}y\,\mathrm{d}x\leqslant \alpha \|v\|_{\mathrm{W}(\Gamma;\gamma)}\text{ for all }v\in\mathrm{W}(\Gamma;\gamma).
    \end{equation*}
\end{proof}
\begin{remark}
    The definition of $w\colon\Gamma\to(0,\infty]$ and $\mathrm{W}(\Gamma;\gamma)$ in \Cref{thm:Tr1} can be generalized. For that reason, let $\widetilde{\Omega}\subset\Omega$ be open and let a measurable functions $\alpha\in\mathrm{L}^{\infty}(\widetilde{\Omega})$ satisfy $\einf_{x\in\widetilde{\Omega}}\left(\int_{\Gamma}\gamma(z,x)\,\mathrm{d}z+\alpha(x)\right)>0$. Instead of assuming that $\int_{\Gamma}\gamma(z,x)\,\mathrm{d}z<\infty$ holds for a\,.e\,. $x\in\Omega$, we just assume that this inequality holds for a\,.e\,. on $\widetilde{\Omega}$. Then our statements remain true, if we define  $w\colon \Gamma \to (0,\infty]$ by
    \begin{equation*}
        w(y)=\int_{\widetilde{\Omega}}\dfrac{\gamma(y,x)}{\int_{\Gamma}\gamma(z,x)\,\mathrm{d}z+\alpha(x)}\,\mathrm{d}x\quad\text{ for }y\in\Gamma
    \end{equation*}
    and set $\mathrm{W}(\Gamma;\gamma):=\{u:\Gamma\to \mathbb{R}\quad\text{measurable with }\|u\|_{\mathrm{W}(\Gamma,\gamma)}<\infty\}$, where
    \begin{equation*}
        \|u\|_{\mathrm{W}(\Gamma;\gamma)}:=\int_{\Gamma}u^{2}(y)w(y)\,\mathrm{d}y+\int_{\Gamma}\int_{\Gamma}(u(y)-u(z))^{2}\int_{\widetilde{\Omega}}\dfrac{\gamma(y,x)\gamma(z,x)}{\int_{\Gamma}\gamma(s,x)\,\mathrm{d}s+\alpha(x)}\,\mathrm{d}x\,\mathrm{d}y\,\mathrm{d}z.
    \end{equation*}
    Furthermore we highlight, that we assume $\Omega\subset\Rd$ to be bounded so that $w$ is integrable on $\Gamma$, and therefore a.e.\;finite. Even if $\lambda(\Omega)=\infty$ holds, \Cref{thm:Tr1} and \Cref{thm:trsur} remain valid if $w$ is finite a.e.\;on $\Gamma$. This is for example the case if $ \int_{\Omega}\gamma(y,x)\,\mathrm{d}x<\infty$ holds for a.e. $y\in\Gamma$.
\end{remark}
\begin{theorem}
    Let $\Omega\subset \Rd$ be a bounded, nonempty, open subset and $\gamma\in\kernel$ such that for a.e.\;$x\in \Omega$ we have
    \begin{align*}
        \int_{\Gamma}\gamma(z,x)\,\mathrm{d}x<\infty. 
    \end{align*}
    For a given $c\in(0,\infty)$, define the function $w\colon \Gamma \to (0,\infty]$ by
    \begin{equation*}
        w(y)=\int_{\Omega}\dfrac{\gamma(y,x)}{\int_{\Gamma}\gamma(z,x)\,\mathrm{d}z+c}\,\mathrm{d}x.
    \end{equation*}
    Furthermore let $\Gamma_{0}\subset\Gamma$ be measurable, then $\Tr\vert_{\Gamma_{0}}\colon\V\to \mathrm{L}^{2}(\Gamma_{0},w),\; u\mapsto u\vert_{\Gamma_{0}}$ is a bounded linear operator and $\{u\in\V\text{ such that } \Tr\vert_{\Gamma_{0}} u=0\}$ is a closed subspace of $\V$ with respect to $\|\cdot\|_{\V}$.
\end{theorem}
\begin{proof}
    Because of \Cref{thm:Tr1} we see that $\Tr\vert_{\Gamma_{0}}$ is a bounded linear operator. Now let $(v_{n})_{n\in\mathbb{N}}$ be a sequence  in $\{u\in\V\colon \Tr\vert_{\Gamma_{0}} u=0\}$ converging to $v\in\V$ with respect to $\|\cdot\|_{\V}$, then
    \begin{equation*}
        \|\Tr\vert_{\Gamma_{0}}(v)\|_{\mathrm{L}^{2}(\Gamma_{0};w)}=\|\Tr\vert_{\Gamma_{0}}(v-v_{n})\|_{\mathrm{L}^{2}(\Gamma_{0};w)}\leqslant \|v-v_{n}\|_{\V}\to 0,\quad \text{ for }n\to\infty.
    \end{equation*}
    Therefore $v(y)w(y)=0$ holds for a.e.\;$y\in\Gamma_{0}$ and because $w$ is positive on $\Gamma_{0}\subset\Gamma$ we get $\Tr\vert_{\Gamma_{0}}(v)=0$.
\end{proof}

\section{The regional problem}\label{sec:reg}

In this section we consider the nonlocal Robin problem
\begin{equation}\tag{RP}\label{RP}
    \left\{\begin{aligned}
        \mathcal{L}u(x)&=f(x)\quad &\text{for }x\in\Omega,\\
        \alpha(y)u(y)w(y)+(1-\alpha(y))\Neu u(y)&=g(y)w(y)\quad &\text{for }y\in\Gamma,
    \end{aligned}\right.
    \end{equation}
for $\gamma\in\kernel$. Recalling \cref{de:weak_s} and \cref{thm:gauss}, we define our weak solution:
\begin{definition}\label{de:weak_srobin}
    Let $\Omega\subset \Rd$ be bounded, nonempty and open and let $\gamma\in\kernel$ be symmetric with $\int_{\Gamma}\gamma(x,y)\,\mathrm{d}y<\infty$ for a.e.\;$x\in\Omega$. Set 
    \begin{equation*}
        w\colon\Gamma\to[0,\infty),\;w(y)=\int_{\Omega}\frac{\gamma(y,x)}{\int_{\Gamma}\gamma(z,x)\,\mathrm{d}z+c}\,\mathrm{d}x,
    \end{equation*}
    where the constant $c\geqslant 0$ satisfies $\einf_{x\in\Omega}\int_{\Gamma}\gamma(z,x)\,\mathrm{d}z+c>0$. Furthermore let $\alpha\colon\Gamma\to[0,1]$ be measurable and set $\Gamma_{1}:=\{y\in\Gamma\colon\alpha(y)<1\}$. Given measurable functions $f\colon \Omega\to\mathbb{R}$ and $g\colon \Gamma\to\mathbb {R}$, a function $u\in \{v\in\V\colon v=g \text{ a.e.\;on }\Gamma\setminus\Gamma_{1}\}$ is called a weak solution to the nonlocal Robin problem \eqref{RP} if
    \begin{align*}
    	\int_{\Omega}f(x)v(x)\,\mathrm{d}x+\int_{\Gamma_{1}}\frac{g(y)\,v(y)\,w(y)}{1-\alpha(y)}\,\mathrm{d}y=\mathfrak{B}(u,v)+\int_{\Gamma_{1}}\frac{\alpha(y)\,u(y)\,v(y)\,w(y)}{1-\alpha(y)}\,\mathrm{d}y
    \end{align*}
    holds for all $v \in \V$.
\end{definition}
Following the proof of \Cref{thm:exnonhom}, we can easily obtain sufficient assumptions, so that the weak solution exists. However we want to present a different approach.

As first shown in \cite{remneum}, for the fractional Laplacian, it is possible to incorporate the following nonlocal Robin boundary condition
\begin{equation}\label{eq:robin-boundary-condition}
    \alpha(y)u(y)+(1-\alpha(y))\left(\int_{\Omega}\frac{1}{\|x-y\|^{d+2s}}\right)^{-1}\int_{\Omega}\frac{u(x)-u(y)}{\|x-y\|^{d+2s}}\,\mathrm{d}x=0,\quad\text{for }y\in\Gamma,
\end{equation}
into the nonlocal operator $\mathcal{L}$, by rearranging \eqref{eq:robin-boundary-condition} into
\begin{align*}
    u(y)=(1-\alpha(y))\left(\int_{\Omega}\frac{1}{\|x-y\|^{d+2s}}\right)^{-1}\int_{\Omega}\frac{u(x)}{\|x-y\|^{d+2s}}\,\mathrm{d}x,\quad\text{for }y\in\Gamma,
\end{align*}
and then inserting this representation into $\mathcal{L}u$. Note that boundary conditions on $\widehat{\Gamma}\setminus\Gamma$ are irrelevant in the evaluation of $\La u$ (see \cref{rem:nons}). In this section, we follow along the same lines and in the following theorem we reformulate our nonlocal Robin problem \eqref{RP}, where we consider the following 
\begin{equation*}
    \alpha(y)u(y)-(1-\alpha(y))\int_{\Omega}u(x)\gamma(x,y)-u(y)\gamma(y,x)\,\mathrm{d}x=g(y)\quad \text{for }y\in\Gamma,
\end{equation*}
Robin boundary condition, for simplicity, because the proof for
\begin{equation*}
    \alpha(y)u(y)w(y)-(1-\alpha(y))\int_{\Omega}u(x)\gamma(x,y)-u(y)\gamma(y,x)\,\mathrm{d}x=g(y)w(y)\quad \text{for }y\in\Gamma,
\end{equation*}
is analogous.
\begin{theorem}\label{thm:reg}
    Let $\Omega\subset\Rd$ be an open, nonempty subset and $\gamma\in\kernel$ with 
    \begin{align*}
        &\|\gamma(\cdot,x)\|_{\mathrm{L}^{\infty}(\Gamma)}+\int_{\Gamma}\gamma(x,y)\,\mathrm{d}y<\infty, \quad\text{for a.e.\;}x\in\Omega\\
        \text{and}\quad &\|\gamma(\cdot,y)\|_{\mathrm{L}^{\infty}(\Omega)}+\int_{\Omega}\gamma(y,x)\,\mathrm{d}x<\infty,\quad \text{for a.e.\;}y\in\Gamma.
    \end{align*}
    Furthermore let the measurable function $u\colon\Rd\to\mathbb{R}$ satisfy
    \begin{align*}
       \int_{\Omega}\vert u(x)\vert\,\mathrm{d}x<\infty \quad\text{and}\quad \int_{\Omega}\int_{\Rd} \vert u(x)\gamma(x,y)-u(y)\gamma(y,x)\vert\,\mathrm{d}y\,\mathrm{d}x<\infty.
    \end{align*}
    Let  $\alpha\colon\Gamma\to[0,1]$ be measurable and $g\in\mathrm{L}^{1}(\Gamma)$ be given, such that for all  $y\in\Gamma$ we have
    \begin{equation*}\tag{B}\label{B}
        \alpha(y)u(y)-(1-\alpha(y))\int_{\Omega}u(x)\gamma(x,y)-u(y)\gamma(y,x)\,\mathrm{d}x=g(y).
    \end{equation*}
    Then defining $\gamma_{\alpha}\colon\Omega\times\Omega\to[0,\infty]$ by
    \begin{equation*}
        \gamma_{\alpha}(x,z)=\gamma(x,z)+\int_{\Gamma}\frac{(1-\alpha(y))\gamma(x,y)\gamma(y,z)}{(1-\alpha(y))\int_{\Omega}\gamma(y,v)\,\mathrm{d}v+\alpha(y)}\,\mathrm{d}y,\quad (x,z)\in\Omega\times\Omega,
    \end{equation*}
    we have for a.e.\;$x\in\Omega$ that 
    \begin{align*}
        \mathcal{L}_{\gamma}u(x)=&\int_{\Rd}u(x)\gamma(x,y)-u(y)\gamma(y,x)\,\mathrm{d}y\\
        =&\int_{\Omega} u(x)\gamma_{\alpha}(x,z)-u(z)\gamma_{\alpha}(z,x)\,\mathrm{d}z\\
        &+u(x)\left(\int_{\widehat{\Gamma}\setminus\Gamma}\gamma(x,y)\,\mathrm{d}y+\int_{\Gamma}\frac{\alpha(y)\gamma(x,y)}{(1-\alpha(y))\int_{\Omega}\gamma(y,z)\,\mathrm{d}z+\alpha(y)}\,\mathrm{d}y\right)\\
        &-\int_{\Gamma}\dfrac{g(y)\gamma(y,x)}{(1-\alpha(y))\int_{\Omega}\gamma(y,v)\,\mathrm{d}v+\alpha(y)}\,\mathrm{d}y.
    \end{align*}
\end{theorem}
\begin{proof}
    First of all by Fubini's Theorem $\gamma_{\alpha}$ is measurable and for a.e.\;$x\in\Omega$ have
    \begin{align*}
        0\leqslant &\int_{\Omega}\int_{\Gamma}\frac{(1-\alpha(y))\gamma(x,y)\gamma(y,z)}{(1-\alpha(y))\int_{\Omega}\gamma(y,v)\,\mathrm{d}v+\alpha(y)}\,\mathrm{d}y\,\mathrm{d}z\\
        \leqslant &\int_{\Gamma}\frac{\int_{\Omega}\gamma(y,z)\,\mathrm{d}z}{\int_{\Omega}\gamma(y,v)\,\mathrm{d}v}\gamma(x,y)\,\mathrm{d}y<\infty.
    \end{align*}
    So $\gamma_{\alpha}$ is finite a.e.\;on $\Omega\times\Omega$. Without loss of generality we assume 
    \begin{equation*}
        \int_{\Rd} \vert u(x)\gamma(x,y)-u(y)\gamma(y,x)\vert\,\mathrm{d}y<\infty
    \end{equation*}
    for all $x\in \Omega$. Now let $x\in\Omega$, and for $y\in\Gamma$ set
    \begin{equation*}
        c(y):=\frac{1}{(1-\alpha(y))\int_{\Omega}\gamma(y,v)\,\mathrm{d}v+\alpha(y)}.
    \end{equation*}
    Then we have to show that  
    \begin{align*}
        \mathcal{L}_{\gamma}u(x)=&\int_{\Omega} u(x)\gamma_{\alpha}(x,z)-u(z)\gamma_{\alpha}(z,x)\,\mathrm{d}z\\
        &+u(x)\left(\int_{\widehat{\Gamma}\setminus\Gamma}\gamma(x,y)\,\mathrm{d}y+\int_{\Gamma}c(y)\alpha(y)\gamma(x,y)\,\mathrm{d}y\right)\\\
        &-\int_{\Gamma}c(y)g(y)\,\mathrm{d}y.
    \end{align*}
    First we observe 
     \begin{align*}
         &\int_{\Omega}\vert u(x)\vert\gamma(x,y)+\vert u(y)\vert\gamma(y,x)\,\mathrm{d}x\\
         \leqslant &\|\gamma(\cdot,y)\|_{\mathrm{L}^{\infty}(\Omega)} \int_{\Omega}\vert u(x)\vert\,\mathrm{d}x+\vert u(y)\vert\int_{\Omega}\gamma(y,x)\,\mathrm{d}x<\infty,\quad \text{for }y\in\Gamma,
     \end{align*}
    and by rearranging \eqref{B} we therefore for $y\in\Gamma$ obtain
    \begin{equation}\tag{B*}\label{B*}
        c(y)\left(g(y)+(1-\alpha(y))\int_{\Omega} u(z)\gamma(z,y)\,\mathrm{d}z\right)=u(y).
    \end{equation}
    Linearity of the integral yields
    \begin{equation*}
        \mathcal{L}_{\gamma}u(x)=\int_{\Omega} u(x)\gamma(x,y)-u(y)\gamma(y,x)\,\mathrm{d}y+\int_{\widehat{\Gamma}} u(x)\gamma(x,y)-u(y)\gamma(y,x)\,\mathrm{d}y,
    \end{equation*}
    and by \eqref{B*} we have
    \begin{align*}
        &\int_{\Gamma} u(x)\gamma(x,y)-u(y)\gamma(y,x)\,\mathrm{d}y\\
        =&\int_{\Gamma}c(y)\left(\frac{1}{c(y)} u(x)\gamma(x,y)-\left(g(y)+(1-\alpha(y))\int_{\Omega} u(z)\gamma(z,y)\,\mathrm{d}z\right)\gamma(y,x)\right)\,\mathrm{d}y.
    \end{align*}
    For $y\in\Gamma$ we have 
    \begin{align*}
        &\frac{1}{c(y)} u(x)\gamma(x,y)-\left((1-\alpha(y))\int_{\Omega} u(z)\gamma(z,y)\,\mathrm{d}z\right)\gamma(y,x)\\
        &=\left((1-\alpha(y))\int_{\Omega}u(x)\gamma(x,y)\gamma(y,z)-u(z)\gamma(z,y)\gamma(y,x)\,\mathrm{d}z\right)+u(x)\alpha(y)\gamma(x,y).
    \end{align*}
    Therefore it remains to show that the following three integrals are all finite
    \begin{align*}
        &\int_{\Gamma}\left\vert c(y)g(y)\gamma(y,x)\right\vert\,\mathrm{d}y,\\
        &\int_{\Gamma}c(y)\left((1-\alpha(y))\int_{\Omega}\left\vert u(x)\gamma(x,y)\gamma(y,z)-u(z)\gamma(z,y)\gamma(y,x)\right\vert\,\mathrm{d}z\right)\,\mathrm{d}y,\\
        \text{and }& \int_{\Gamma}\left\vert c(y)\alpha(y)\gamma(x,y)\right\vert\,\mathrm{d}y.
    \end{align*}
    Because of $\frac{1}{c(y)}\geqslant \min \{(1-\alpha(y))\int_{\Omega}\gamma(y,v)\,\mathrm{d}v,\,\alpha(y) \}$ for $y\in\Gamma$, we get
    \begin{equation*}
        \int_{\Gamma}c(y)\alpha(y) \gamma(x,y)\,\mathrm{d}y\leqslant\int_{\Gamma}\gamma(x,y)\,\mathrm{d}y<\infty.
    \end{equation*}
    And since $\int_{\Omega}\int_{\Gamma}\frac{\vert g(y)\vert\gamma(y,x)}{\int_{\Omega}\gamma(y,v)\,\mathrm{d}v}\,\mathrm{d}y\,\mathrm{d}x=\int_{\Gamma}\vert g(y)\vert\mathrm{d}y<\infty$ holds, we obtain
    \begin{align*}
        &\int_{\Gamma}\left\vert c(y)g(y)\gamma(y,x)\right\vert\,\mathrm{d}y\\
        \leqslant&\int_{\Gamma}c(y)(1-\alpha(y))\vert g(y)\vert\gamma(y,x)\,\mathrm{d}y+\int_{\Gamma}c(y)\alpha(y)\vert g(y)\vert\gamma(y,x)\,\mathrm{d}y\\
        \leqslant&\int_{\Gamma}\dfrac{\vert g(y)\vert\gamma(y,x)}{\int_{\Omega}\gamma(y,v)\,\mathrm{d}v}\,\mathrm{d}y+\int_{\Gamma}\vert g(y)\vert\gamma(y,x)\,\mathrm{d}y\\
        \leqslant&\int_{\Gamma}\dfrac{\vert g(y)\vert\gamma(y,x)}{\int_{\Omega}\gamma(y,v)\,\mathrm{d}v}\,\mathrm{d}y+\int_{\Gamma}\vert g(y)\vert\,\mathrm{d}y\|\gamma(\cdot,x)\|_{\mathrm{L}^{\infty}(\Gamma)}\\
        <&\infty.
    \end{align*}
    Finally we have for a.e.\;$(y,z)\in\Gamma\times\Omega$ that
    \begin{align*}
        &\left\vert u(x)\gamma(x,y)\gamma(y,z)-u(z)\gamma(z,y)\gamma(y,x)\right\vert\\
        =&\left\vert u(x)\gamma(x,y)\gamma(y,z)-u(y)\gamma(y,x)\gamma(y,z)+u(y)\gamma(y,x)\gamma(y,z)-u(z)\gamma(z,y)\gamma(y,x)\right\vert\\
        \leqslant&\left\vert u(x)\gamma(x,y)-u(y)\gamma(y,x)\right\vert\gamma(y,z)+\left\vert u(y)\gamma(y,z)-u(z)\gamma(z,y)\right\vert\gamma(y,x)
    \end{align*}
    and
    \begin{equation*}
        c(y)(1-a(y))\leqslant\frac{1}{\int_{\Omega}\gamma(y,v)\,\mathrm{d}v}
    \end{equation*}
    holds. Due to
    \begin{align*}
        &\int_{\Omega}\int_{\Gamma}c(y)\left((1-\alpha(y))\int_{\Omega}\left\vert u(x)\gamma(x,y)\gamma(y,z)-u(z)\gamma(z,y)\gamma(y,x)\right\vert\,\mathrm{d}z\right)\,\mathrm{d}y\,\mathrm{d}x\\
        \leqslant &2 \int_{\Omega}\int_{\Gamma}\left\vert u(x)\gamma(x,y)-u(y)\gamma(y,x)\right\vert\,\mathrm{d}y\,\mathrm{d}x<\infty.
    \end{align*}
    we, without loss of generality, conclude 
    \begin{equation*}
        \int_{\Gamma}c(y)\left((1-\alpha(y))\int_{\Omega}\left\vert u(x)\gamma(x,y)\gamma(y,z)-u(z)\gamma(z,y)\gamma(y,x)\right\vert\,\mathrm{d}z\right)\,\mathrm{d}y<\infty.
    \end{equation*}
    The rest follows then by Fubini's Theorem.
\end{proof} 

We now exploit the result of \Cref{thm:reg} to reformulate the nonlocal Robin problem \eqref{RP} into the equivalent \textit{regional problem}
\begin{equation*}\tag{REG}\label{reg}
    \left\{\mathcal{L}_{\gamma_{\alpha}}u(x)+\gamma_{\alpha,\Omega}(x)u(x)=f(x)+g_{\Gamma}(x) \quad\text{for }x\in\Omega,\right.
\end{equation*}
where for $(x,z)\in\Omega\times\Omega$ we set
\begin{align*}
      g_{\Gamma}(x)&:= \displaystyle{\int_{\Gamma}\dfrac{g(y)\gamma(y,x)}{(1-\alpha(y))\int_{\Omega}\gamma(y,v)\,\mathrm{d}v+\alpha(y)}\,\mathrm{d}y},\\
     \gamma_{\alpha}(x,z)&:=\gamma(x,z)+\int_{\Gamma}\frac{(1-\alpha(y))\gamma(x,y)\gamma(y,z)}{(1-\alpha(y))\int_{\Omega}\gamma(y,v)\,\mathrm{d}v+\alpha(y)}\,\mathrm{d}y\\
     \text{and}\quad\gamma_{\alpha,\Omega}(x)&:=\displaystyle{\int_{\Gamma}\frac{\alpha(y)\gamma(x,y)}{(1-\alpha(y))\int_{\Omega}\gamma(y,z)\,\mathrm{d}z+\alpha(y)}\,\mathrm{d}y}+\displaystyle{\int_{\widehat{\Gamma}\setminus\Gamma}\gamma(x,y)\,\mathrm{d}y}.
\end{align*}
We note that the regional problem \eqref{reg} is of the form
\begin{equation*}
    \left\{\mathcal{L}_{\eta}u(x)+\lambda(x)u(x)=\widetilde{f}(x)\quad\text{for }x\in\Omega,\right.
\end{equation*}
where $\eta\in\kernel$ vanishes identically outside $\Omega\times\Omega$ and both $\lambda\colon\Omega\to[0,\infty)$ and $\widetilde{f}\colon\Omega\to\mathbb{R}$ are measurable. We recapitulate:

\begin{definition}
     Let $\Omega\subset\Rd$ be an open, nonempty subset and $\gamma\in\kernel$ with $\widehat{\Gamma}=\Gamma$,
    \begin{align*}
        &\|\gamma(\cdot,x)\|_{\mathrm{L}^{\infty}(\Gamma)}+\int_{\Gamma}\gamma(x,y)\,\mathrm{d}y<\infty, \quad\text{for }x\in\Omega\\
        \text{and}\quad &\|\gamma(\cdot,y)\|_{\mathrm{L}^{\infty}(\Omega)}+\int_{\Omega}\gamma(y,x)\,\mathrm{d}x<\infty,\quad \text{for }y\in\Gamma.
    \end{align*}
    Let $f\colon \Omega\to\mathbb{R}$ and $g\colon \Gamma\to\mathbb {R}$ be given and set, as before, $ g_{\Gamma}$, $\gamma_{\alpha}$ and $\gamma_{\alpha,\Omega}$. Then the measurable function $u\colon \Omega\to\mathbb{R}$ is a regional solution to the Robin problem \eqref{RP}, if $u$ is a solution of \eqref{reg}.
    
    And $u\in\mathrm{V}(\Omega;\gamma_{\alpha})$ is a weak regional solution to the Robin problem \eqref{RP}, if $u$ is a weak solution (in the sense of \Cref{de:weak_s}) of the regional problem \eqref{reg}.
\end{definition}

In \Cref{rem:trivial_neum_case} and \Cref{thm:nonsym} well-posedness results regarding problem \eqref{reg} are given. Furthermore, for the time--dependent case the regional solution has also been studied by Cortazar et al.\,\cite{CORTAZAR2007360}. Because $\gamma_{\alpha,\Omega}$ is in general not an element of $\mathrm{L}^{2}(\Omega)$, we now study a slightly different test function space.
\begin{theorem}\label{thm:reg_hilb_hilf}
     Let $\Omega\subset\Rd$ be an open, nonempty subset, $\lambda\colon\Omega\to[0,\infty)$ be a measurable function and let  $\gamma\in\kernel$ be vanishing identically outside $\Omega\times\Omega$. Then $\{v\in\V\colon \int_{\Omega}v^{2}(x)\lambda(x)\,\mathrm{d}x<\infty\}$ is a Hilbert space with respect to the inner product
     \begin{equation*}
        \langle u,v\rangle_{1}=\int_{\Omega}u(x)v(x)(1+\lambda(x))\,\mathrm{d}x+\int_{\Omega}\int_{\Omega}(u(x)-u(y))(v(x)-v(y))\gamma(y,x)\,\mathrm{d}y\,\mathrm{d}x.
     \end{equation*}
\end{theorem}
\begin{proof}
    We now show that the norm induced by the inner product is complete. So let $(u_{n})_{n\in\mathbb{N}}$ be a Cauchy sequence in  $\{v\in\V\colon \int_{\Omega}v^{2}(x)\lambda(x)\,\mathrm{d}x<\infty\}$. Then $(u_{n})_{n\in\mathbb{N}}$ is Cauchy in $\V$, and by \Cref{corollary:hilb} converges to $u\in\V$ with respect to $\|\cdot\|_{\V}$. Without loss of generality we assume that $u_{n}$ converges a.e.\;in $\Omega$ to $u$. Therefore the Lemma of Fatou yields
    \begin{equation*}
        \int_{\Omega}u^{2}(x)\lambda(x)\,\mathrm{d}x\leqslant\liminf_{n\to\infty}\int_{\Omega}u_{n}^{2}(x)\lambda(x)\,\mathrm{d}x\leqslant \sup_{n\in\mathbb{N}}\int_{\Omega}u_{n}^{2}(x)\lambda(x)\,\mathrm{d}x<\infty.
    \end{equation*}
    Now that we have $u\in\{v\in\V\colon \int_{\Omega}v^{2}(x)\lambda(x)\,\mathrm{d}x<\infty\} $, we use the Lemma of Fatou again to obtain
    \begin{equation*}
        \lim_{n\to\infty}\int_{\Omega}(u(x)-u_{n}(x))^{2}\lambda(x)\,\mathrm{d}x\leqslant\lim_{n\to\infty}\liminf_{m\to\infty}\int_{\Omega}(u_{m}(x)-u_{n}(x))^{2}\lambda(x)\,\mathrm{d}x=0.
    \end{equation*}
    All in all we get that $\{v\in\V\colon \int_{\Omega}v^{2}(x)\lambda(x)\,\mathrm{d}x<\infty\}$ is a Hilbert space with respect to the inner product
     \begin{equation*}
        \langle u,v\rangle_{1}=\int_{\Omega}u(x)v(x)(1+\lambda(x))\,\mathrm{d}x+\int_{\Omega}\int_{\Omega}(u(x)-u(y))(v(x)-v(y))\gamma(y,x)\,\mathrm{d}y\,\mathrm{d}x
     \end{equation*}
     for $u,v \in \{v\in\V\colon \int_{\Omega}v^{2}(x)\lambda(x)\,\mathrm{d}x<\infty\}$.
\end{proof}
Finally we compare our new test function space with $\V$.
\begin{theorem}
    Let $\Omega\subset\Rd$ be an open, nonempty subset and $\gamma\in\kernel$ satisfy $\widehat{\Gamma}=\Gamma$,
    \begin{align*}
        &\|\gamma(\cdot,x)\|_{\mathrm{L}^{\infty}(\Gamma)}+\int_{\Gamma}\gamma(x,y)\,\mathrm{d}y<\infty \quad\text{for }x\in\Omega \\
        \text{and } &\|\gamma(\cdot,y)\|_{\mathrm{L}^{\infty}(\Omega)}+\int_{\Omega}\gamma(y,x)\,\mathrm{d}x<\infty\quad \text{for }y\in\Gamma.
    \end{align*}
    Furthermore let $\alpha\colon\Gamma\to[0,1]$ be measurable and set 
    \begin{equation*}
        \mathrm{V}_{1-\alpha}(\Omega;\gamma):=\{v\in\mathrm{V}(\Omega;\gamma_{\alpha})\colon \int_{\Omega}v^{2}(x)\gamma_{\alpha,\Omega}(x)\,\mathrm{d}x<\infty\},
    \end{equation*}
    where 
    \begin{align*}
         \gamma_{\alpha}(x,z)&:=\gamma(x,z)+\int_{\Gamma}\frac{(1-\alpha(y))\gamma(x,y)\gamma(y,z)}{(1-\alpha(y))\int_{\Omega}\gamma(y,v)\,\mathrm{d}v+\alpha(y)}\,\mathrm{d}y&\\
        \text{and}\quad\gamma_{\alpha,\Omega}(x)&:=\displaystyle{\int_{\Gamma}\frac{\alpha(y)\gamma(x,y)}{(1-\alpha(y))\int_{\Omega}\gamma(y,z)\,\mathrm{d}z+\alpha(y)}\,\mathrm{d}y}\quad &
    \end{align*}
    for $(x,z)\in\Omega\times\Omega$. Then:
    \begin{enumerate}[label=(\roman*)]
        \item We have $\gamma_{1}=\gamma$ in $\Omega\times\Omega$  and $\gamma_{0,\Omega}=0$ in $\Omega$.
        \item $\mathrm{V}_{1-\alpha}(\Omega;\gamma)$ is a Hilbert space with respect to the inner product
        \begin{align*}
            \langle u,v\rangle_{\mathrm{V}_{1-\alpha}(\Omega;\gamma)}:=&\int_{\Omega}u(x)v(x)(1+\gamma_{\alpha,\Omega}(x))\,\mathrm{d}x\\
            &+\int_{\Omega}\int_{\Omega}(u(x)-u(y))(v(x)-v(y))\gamma_{\alpha}(y,x)\,\mathrm{d}y\,\mathrm{d}x.
        \end{align*}
        \item $\Vd$ and $\{v\in\V\colon \Tr(v)=0\}$ are isomorphic, in other words there is a bijective and bounded operator $T\colon\Vd\to\{v\in\V\colon \Tr(v)=0\}$ with
        \begin{equation*}
            \langle u,v\rangle_{\Vd}=\langle T u, T v\rangle_{\V},\quad \text{ for } u,v\in\Vd.
        \end{equation*}
        \item $\mathrm{V}_{1}(\Omega;\gamma)$ is continuously embedded in $\V$, namely there is a constant $C>0$ such that for any $u\in\mathrm{V}_{1}(\Omega;\gamma)$ there exists $\widetilde{u}\in\V$ with $u=\widetilde{u}$ in $\Omega$ and
        \begin{equation*}
            \|\widetilde{u}\|_{\V}\leqslant C  \|u\|_{\mathrm{V}_{1}(\Omega;\gamma)}.
        \end{equation*}
        \item $\V$  is continuously embedded in $\mathrm{V}_{1}(\Omega;\gamma)$, namely there is a constant $C>0$ such that for any $u\in\V$ we have $u\vert_{\Omega}\in \mathrm{V}_{1}(\Omega;\gamma)$ and
        \begin{equation*}
            \|u\|_{\mathrm{V}_{1}(\Omega;\gamma)}\leqslant C \|u\|_{\V} .
        \end{equation*}
    \end{enumerate}
\end{theorem}
\begin{proof}
    While $(i)$ follows by definition, we obtain $(ii)$ as a consequence of \Cref{thm:reg_hilb_hilf}. 
    
    In order to show $(iii)$, we define the zero extension operator outside $\Omega$,
    \begin{equation*}
        E\colon\Vd\to\{v\in\V\colon \Tr (v)=0\},\,Eu (x)=\begin{cases} u(x),\quad &\text{for } x\in\Omega,\\0,\quad &\text{for } x\in\Gamma.\end{cases}
    \end{equation*}
    Then $E$ is a bijective operator with 
    \begin{align*}
            &\langle u,v\rangle_{\Vd}\\
            =&\int_{\Omega}u(x)v(x)(1+\gamma_{1,\Omega}(x))\,\mathrm{d}x+\int_{\Omega}\int_{\Omega}(u(x)-u(y))(v(x)-v(y))\gamma(y,x)\,\mathrm{d}y\,\mathrm{d}x&\\
            =&\int_{\Omega}E u(x)E v(x)\,\mathrm{d}x+\int_{\Omega}\int_{\Rd}(E u(x)-E u(y))(E v(x)-E v(y))\gamma(y,x)\,\mathrm{d}y\,\mathrm{d}x&\\
            =&\langle E u, E v\rangle_{\V}
    \end{align*}
    \text{for all } $u,v\in\Vd$.
    For $u\in\mathrm{V}_{1}(\Omega;\gamma) $ we define
    \begin{equation*}
        \widetilde{u}(x)=\begin{cases} u(x),\quad &\text{for } x\in\Omega,\\ \displaystyle{\int_{\Omega}\frac{u(z)\gamma(z,x)}{\int_{\Omega}\gamma(v,x)\,\mathrm{d}v}\,\mathrm{d}z},\quad &\text{for } x\in\Gamma.\end{cases}
    \end{equation*}
    Because for a.e.\;$x\in\Gamma$ we obtain
    \begin{equation*}
        \int_{\Omega}\frac{\vert u(z)\vert\gamma(z,x)}{\int_{\Omega}\gamma(v,x)\,\mathrm{d}v}\,\mathrm{d}z\leqslant\int_{\Omega}\frac{u^{2}(z)\gamma(z,x)}{\int_{\Omega}\gamma(v,x)\,\mathrm{d}v}\,\mathrm{d}z\leqslant \frac{\|u\|_{\mathrm{L}^{2}(\Omega)}\|\gamma(\cdot,x)\|_{\mathrm{L}^{\infty}(\Omega)}}{\int_{\Omega}\gamma(v,x)\,\mathrm{d}v}<\infty,
    \end{equation*}
    the extension $\widetilde{u}$ is well defined. Then we get by Hölder's inequality and Fubini's Theorem, that 
    \begin{align*}
        &\int_{\Omega}\int_{\Gamma}(\widetilde{u}(x)-\widetilde{u}(y))^{2}\gamma(y,x)\,\mathrm{d}y\,\mathrm{d}x\\
        =&\int_{\Omega}\int_{\Gamma}\left(u(x)-\int_{\Omega}\frac{u(z)\gamma(z,y)}{\int_{\Omega}\gamma(v,y)\,\mathrm{d}v}\,\mathrm{d}z\right)^{2}\gamma(y,x)\,\mathrm{d}y\,\mathrm{d}x\\
        =&\int_{\Omega}\int_{\Gamma}\left(\int_{\Omega}\frac{(u(x)-u(z))\gamma(z,y)}{\int_{\Omega}\gamma(v,y)\,\mathrm{d}v}\,\mathrm{d}z\right)^{2}\gamma(y,x)\,\mathrm{d}y\,\mathrm{d}x\\
        \leqslant&\int_{\Omega}\int_{\Omega}(u(x)-u(z))^{2}\int_{\Gamma}\frac{\gamma(y,x)\gamma(z,y)}{\int_{\Omega}\gamma(v,y)\,\mathrm{d}v}\,\mathrm{d}y\,\mathrm{d}z\,\mathrm{d}x
    \end{align*}
    holds and we therefore conclude
    \begin{align*}
        \|\widetilde{u}\|_{\V}^{2}&=\int_{\Omega}\widetilde{u}^{2}(x)\,\mathrm{d}x+\int_{\Omega}\int_{\Rd}(\widetilde{u}(x)-\widetilde{u}(y))^{2}\gamma(y,x)\,\mathrm{d}y\,\mathrm{d}x\\
        &\leqslant \int_{\Omega}u^{2}(x)\,\mathrm{d}x+\int_{\Omega}\int_{\Omega}(u(x)-u(y))^{2}\gamma_{0}(y,x)\,\mathrm{d}y\,\mathrm{d}x
        = \|u\|_{\mathrm{V}_{1}(\Omega;\gamma)}^{2}.
    \end{align*}
    For $v\in\V$ we have by Jensen's inequality
    \begin{align*}
        &\int_{\Omega}\int_{\Omega}(v(x)-v(z))^{2}\int_{\Gamma}\frac{\gamma(y,x)\gamma(z,y)}{\int_{\Omega}\gamma(v,y)\,\mathrm{d}v}\,\mathrm{d}y\,\mathrm{d}z\,\mathrm{d}x\\
        =&\int_{\Omega}\int_{\Omega}\int_{\Gamma}(v(x)-v(y)+v(y)-v(z))^{2}\frac{\gamma(y,x)\gamma(z,y)}{\int_{\Omega}\gamma(v,y)\,\mathrm{d}v}\,\mathrm{d}y\,\mathrm{d}z\,\mathrm{d}x\\
        \leqslant &2\int_{\Omega}\int_{\Gamma}(v(x)-v(y))^{2}\gamma(y,x)\,\mathrm{d}y\,\mathrm{d}x
    \end{align*}
    and therefore $\|v\|_{\mathrm{V}_{1}(\Omega;\gamma)}\leqslant 2 \|v\|_{\V}$.
\end{proof}

\bibliographystyle{siamplain}
\bibliography{references}
\end{document}